\documentclass[12pt,reqno]{amsart}

\usepackage{amssymb,hyperref}
\usepackage[usenames,dvipsnames]{xcolor}

\newtheorem{theorem}{Theorem}[section]
\newtheorem{proposition}[theorem]{Proposition}
\newtheorem{lemma}[theorem]{Lemma}
\newtheorem{corollary}[theorem]{Corollary}

\theoremstyle{definition}
\newtheorem{definition}[theorem]{Definition}
\newtheorem{remark}[theorem]{Remark}
\newtheorem{assumption}[theorem]{Assumption}

\numberwithin{equation}{section}

\begin{document}

\baselineskip=15pt

\title[Branched projective structures and logarithmic connections]{Branched 
projective structures on a Riemann surface and logarithmic connections}

\author[I. Biswas]{Indranil Biswas}

\address{School of Mathematics, Tata Institute of Fundamental
Research, Homi Bhabha Road, Mumbai 400005, India}

\email{indranil@math.tifr.res.in}

\author[S. Dumitrescu]{Sorin Dumitrescu}

\address{Universit\'e C\^ote d'Azur, CNRS, LJAD, France}

\email{dumitres@unice.fr}

\author[S. Gupta]{Subhojoy Gupta}

\address{Department of Mathematics, Indian Institute of Science, Bangalore 560012, India}

\email{subhojoy@iisc.ac.in}



\date{}
\begin{abstract}
We study the set ${\mathcal P}_S$ consisting of all branched holomorphic projective structures on a
compact Riemann surface $X$ of genus $g
\,\geq\, 1$ and with a fixed branching divisor $S\, :=\, \sum_{i=1}^d n_i\cdot x_i$, where $x_i \,\in\, X$. Under the hypothesis
that $n_i\,=\,1$, for all $i$, with $d$ a positive even integer such that $d \,\neq\, 2g-2$, we show that ${\mathcal P}_S$
coincides with a subset of the set of all logarithmic connections with singular locus $S$, satisfying certain geometric conditions, on
the rank two holomorphic jet bundle $J^1(Q)$, where $Q$ is a fixed holomorphic
line bundle on $X$ such that $Q^{\otimes 2}\,=\, TX\otimes {\mathcal O}_X(S)$. 
The space of all logarithmic connections of the above type is an affine space over the vector
space $H^0(X, K^{\otimes 2}_X \otimes {\mathcal O}_X(S))$ of dimension $3g-3+d$. We conclude that
${\mathcal P}_S$ is a subset of this affine space that has codimenison $d$ at a generic point. 
\end{abstract}

\maketitle

\tableofcontents

\section{Introduction}

A (holomorphic) projective structure on a Riemann surface $X$ is a holomorphic atlas with 
local coordinates in the projective line ${\mathbb C}{\mathbb P}^1$ such that the transition 
maps are restrictions of elements in the M\"obius group $\text{PGL}(2,{\mathbb C})$. Such 
structures arise naturally in the study of second order ordinary differential equations and 
had a major role in the understanding of uniformization theorem for Riemann surfaces (see, for 
example, \cite{Gu} or Chapter VIII of \cite{St}). More precisely, the uniformization theorem 
asserts that any Riemann surface admits a holomorphic projective structure such that the 
corresponding developing map, which is a holomorphic map from the universal cover 
$\widetilde{X}$ of $X$ to ${\mathbb C}{\mathbb P}^1$, is an embedding (with image a round 
unitary disk, if $X$ is compact of genus $g\,\geq\, 2$).

A compact Riemann surface $X$ with genus $g\,\geq\, 2$ is thus uniformized as a quotient of the Poincar\'e's upper-half plane by a Fuchsian group. Using the corresponding projective structure on $X$ as the base point, 
the space of holomorphic projective structures on $X$ is naturally identified with the space of holomorphic sections
$\text{H}^0(X,\, K^{\otimes 2}_X)$ of the 
square of the canonical bundle $K_X$ of $X$. This space of holomorphic quadratic differentials has complex dimension 
$3g-3$ \cite{Gu}.

A more flexible notion of projective structure, which is stable under pull-back through 
ramified covers (and not just \'etale covers), is that of a {\it branched} (holomorphic) 
projective structure, introduced and studied by Mandelbaum in \cite{M1, M2}. A branched 
projective structure is defined by a holomorphic atlas with local charts being finite branched 
coverings of open subsets in ${\mathbb C}{\mathbb P}^1$, while the transition maps are 
restrictions of elements in the M\"obius group $\text{PGL}(2,{\mathbb C})$ (its definition is 
recalled in Section \ref{prelim}). Away from the branching divisor $S$, the Riemann surface 
$X$ inherits a holomorphic projective structure in the classical sense. These structures 
arise, for example, in the study of conical hyperbolic structures on Riemann surfaces or in 
the theory of codimension one transversally projective holomorphic foliations (see, for 
example, \cite{CDF}).

Geometrically, a holomorphic projective structure on a Riemann surface $X$ is known to be given by a flat ${\mathbb 
C}{\mathbb P}^1$-bundle over $X$ endowed with a holomorphic section which is transverse to the
horizontal distribution defining the flat structure (this 
section is also called the developing map of the projective structure). In the branched case, this section fails to be 
transverse to the horizontal distribution exactly at points in the branching divisor $S$ (see, for 
example, \cite{CDF, GKM, LM} or Section \ref{prelim} here).

An unbranched projective structure on a compact Riemann surface $X$ produces a flat 
holomorphic connection on a certain rank two holomorphic vector bundle on $X$. This 
holomorphic vector bundle does not depend on the projective structure: it is the 
unique nontrivial extension of $K^{-1/2}_X$ by $K^{1/2}_X$. The projectivization of this flat rank two 
vector bundle yields the flat ${\mathbb C}{\mathbb P}^1$--bundle over $X$ mentioned above. 
Given a branched projective structure on $X$, we again have a flat holomorphic connection on a 
holomorphic ${\mathbb C}{\mathbb P}^1$--bundle over $X$. But now this
${\mathbb C}{\mathbb P}^1$--bundle, and hence the corresponding
holomorphic rank two vector bundle, in general 
depends on the branched projective structure. The main result in this article is that a branched projective structure produces a 
holomorphic vector bundle of rank two equipped with a logarithmic connection, such that
the holomorphic vector bundle is 
independent of the choice of the branched projective structure, as long as the branching 
divisor is fixed.

More precisely, fix a Riemann surface $X$ with genus $g \,\geq\, 1$ and an effective divisor $S\, :=\, \sum_{i=1}^d 
n_i\cdot x_i$ on $X$. Consider ${\mathcal P}_S$, the space
all branched holomorphic projective structures on $X$ with $S$ as the branching 
divisor. To ease our notation, we work under the simplifying assumption that

\begin{assumption}\label{a0}
Each $n_i\,=\,1$ and $d\,=\, \# S$ is an even integer such that $d \,\neq\, 2g-2$. 
\end{assumption}

Fix a holomorphic line bundle $Q$ on $X$ such that $Q^{\otimes 2}\,=\, TX\otimes {\mathcal 
O}_X(S)$.

Our first main result (Theorem \ref{thm1}) shows that a branched projective structure on $X$ with branching divisor $S$ 
is the same data as a logarithmic connection $D^1$ on the (rank two) first jet bundle $J^1(Q)$ with
singular locus $S$ (so $D^1$ is nonsingular over 
$X\setminus S$) that satisfies certain geometric conditions. The conditions in question are:
\begin{enumerate}
\item the residue of $D^1$ at each 
point of $S$ has eigen-values $0$ and $-1$,

\item the eigen-space for the eigen-value $-1$ is the line given by the kernel
of the natural projection $J^1(Q)\, \longrightarrow\, J^0(Q)\,=\, Q$,

\item for each point $y\, \in\, S$, the homomorphism $\rho(D^1, y)$ in Proposition \ref{propn1} (from the eigen-space of $\text{Res}(D^1,y)$ for the eigen-value
$0$ to the eigen-space of $\text{Res}(D^1,y)$ for the eigen-value $-1$ tensored with the fiber $(K_X)_y$) vanishes; this
 is equivalent to the condition that the local monodromy of $D^1$ around $y$ is trivial (see Proposition
\ref{propn1}), and

\item the logarithmic connection on $\bigwedge^2 J^1(Q)\,=\, {\mathcal O}_X(S)$ coincides with the one given by the de Rham 
differential.
\end{enumerate}

Let ${\mathcal C}(Q)$ denote the space of all logarithmic connections satisfying the first, second and fourth conditions
(note that condition three is omitted). So the elements of ${\mathcal C}(Q)$ that satisfy the third condition
produce branched projective structure on $X$ with branching divisor $S$.
However, two different elements of this subset of ${\mathcal C}(Q)$ can produce the
same branched projective structure on $X$ with branching divisor $S$. Indeed, two elements of this subset of
${\mathcal C}(Q)$ that differ by an automorphism of $J^1(Q)$ produce the same
branched projective connection (see Remark \ref{remnu}). This ambiguity is removed by identifying a special class
of logarithmic connections, which we now explain.

Given a logarithmic connections on $J^1(Q)$, singular over $S$, there is a homomorphism
$J^1(Q)\, \longrightarrow\,Q\otimes K_X\otimes {\mathcal O}_X(S)$ (see Lemma \ref{lemext}). An element of
${\mathcal C}(Q)$ is called {\it special} if the corresponding
homomorphism $J^1(Q)\, \longrightarrow\,Q\otimes K_X\otimes {\mathcal O}_X(S)$ vanishes identically (see Definition \ref{def1}). Let
${\mathcal C}^0(Q)\, \subset\, {\mathcal C}(Q)$ be the locus of special logarithmic connections.
The space of ${\mathcal P}_S$ of all branched projective structures with
branching divisor $S$ is in bijection with the subset of ${\mathcal C}^0(Q)$ defined by the special 
logarithmic connections that satisfy the third of the above four conditions (Corollary \ref{cor4}). This is
proved by showing that the composition
$$
{\mathcal C}^0(Q) \, \hookrightarrow\, {\mathcal C}(Q) \, \longrightarrow\, {\mathcal C}(Q)/\text{Aut}(J^1(Q))
$$
is a bijection, where $\text{Aut}(J^1(Q))$ is the group of all holomorphic automorphisms of $J^1(Q)$ that act
trivially on $\bigwedge^2 J^1(Q)$.

Summarizing:
there is a natural surjective map to ${\mathcal P}_S$ from the space of logarithmic connections on $J^1(Q)$ singular over $S$
satisfying the above four conditions; the restriction of this map to the subset of special connections is both injective and
surjective.

On the other hand, the space of all special connections ${\mathcal C}^0(Q)$ is an affine space over the vector space of
$H^0(X,\, K^{\otimes 2}_X\otimes {\mathcal O}_X(S))$ (Proposition \ref{prop3} and Corollary \ref{cor3}).
As a consequence, we obtain the following result:

\begin{theorem}
Let $X$ be a Riemann surface of genus $g\,\geq\, 1$, and $S$ be a divisor of degree $d$ satisfying Assumption
\ref{a0}. Then the space of branched projective structures ${\mathcal P}_S$ on $X$ is canonically a subset of codimension $d$ of an affine space over the vector space $H^0(X, K^{\otimes 2}_X\otimes {\mathcal O}_X(S))$ of holomorphic quadratic differentials with at most simple poles over $S$.
\end{theorem}

This generalizes Theorem 3 of \cite{M1} which handled the case when the degree of the divisor does not exceed $2g-2$.

The structure of the paper is as follows. Section \ref{s2} introduces the main definitions 
and presents the geometric description of elements of ${\mathcal P}_S$ as flat ${\mathbb 
C}{\mathbb P}^1$--bundles over $X$ endowed with a generically transverse section. Section \ref{s3} gives an 
equivalent linear description of elements in ${\mathcal P}_S$ as flat rank two holomorphic 
vector bundles with a special line subbundle $L$ of degree $g-1-\frac{d}{2}$. From this 
view-point $S$ appears as the divisor of the second fundamental form of $L$ with respect to 
the flat connection (see Lemma \ref{lem1}). In Section \ref{sec.l}, we first recall the definitions of logarithmic
connections and their residue. Proposition \ref{propn1} plays a crucial role in relating branched projective structures
with logarithmic connections. Section \ref{s4} contains the proof of Theorem 
\ref{thm1} which describes elements in ${\mathcal P}_S$ as logarithmic connections on the jet 
bundle $J^1(Q)$ satisfying some specific residue conditions at points in $S$. In Section 
\ref{s5} we show that we can construct such logarithmic connections on the jet bundle 
$J^1(Q)$, prescribing the previous residue conditions at $S$. In Section 
\ref{s6} we define a special class of logarithmic connections on $J^1(Q)$ which parametrizes 
${\mathcal P}_S$ bijectively (Definition \ref{def1} and Corollary \ref{cor4}). This implies that the space 
of these logarithmic connections is naturally identified with an affine space over $H^0(X, 
K^{\otimes 2}_X \otimes {\mathcal O}_X(S))$ (Proposition \ref{prop3}). Section \ref{se7} 
identifies ${\mathcal P}_S$ with a subspace of second order differential operators satisfying 
some natural geometric conditions (Lemma \ref{lem4}).

\section{Preliminaries}\label{s2}

\subsection{Branched projective structure}\label{prelim} 

Let $X$ be a connected Riemann surface. Fix a nonempty finite subset
$$
S_0\, :=\, \{x_1, \, \cdots ,\, x_d\}\, \subset\, X
$$
of $d$ distinct points. For each $x_i$, $1\, \leq\, i\, \leq\, d$, fix an integer $n_i\, \geq\, 1$.
Let
\begin{equation}\label{e0}
S\, :=\, \sum_{i=1}^d n_i\cdot x_i
\end{equation}
be an effective divisor on $X$.

The group of all holomorphic automorphisms of ${\mathbb C}{\mathbb P}^1$ is the M\"obius group
$\text{PGL}(2,{\mathbb C})$. Any
$$
\begin{pmatrix}
a & b\\
c& d
\end{pmatrix} \, \in\, \text{PGL}(2,{\mathbb C})
$$
acts on ${\mathbb C}{\mathbb P}^1\,=\, {\mathbb C}\cup \{\infty\}$ as
$z\, \longmapsto\, \frac{az+b}{cz+d}$.

A branched projective structure on $X$ with branching type $S$ (defined in \eqref{e0})
is given by data $\{(U_j,\, \phi_j)\}_{j\in J}$, where
\begin{enumerate}
\item $U_j\, \subset\, X$ is a connected open subset with $\# (U_j\bigcap S_0)\, \leq\, 1$ such that
$\bigcup_{j\in J} U_j\,=\, X$,

\item $\phi_j\, :\, U_j\,\longrightarrow\, {\mathbb C}{\mathbb P}^1$ is a holomorphic map
which is an immersion on the complement $U_j\setminus (U_j\bigcap S_0)$,

\item if $U_j\bigcap S_0\,=\, x_i$, then $\phi_j$ is of degree $n_i+1$ and totally ramified at $x_i$,
while $\phi_j$ is an embedding if $U_j\bigcap S_0\,=\, \emptyset$, and

\item for every $j,\, j'\, \in\, J$ and every connected component $U$ of $U_j\bigcap U_{j'}$ there is an
element $f_{j,j'}\, \in \, \text{PGL}(2,{\mathbb C})$, such that $\phi_{j}\, =\, f_{j,j'}\circ
\phi_{j'}$ on $U$.
\end{enumerate}

Two data $\{(U_j,\, \phi_j)\}_{j\in J}$ and $\{(U'_j,\, \phi'_j)\}_{j\in J'}$ satisfying the above
conditions are called \textit{equivalent} if their union $\{(U_j,\, \phi_j)\}_{j\in J}
\bigcup \{(U'_j,\, \phi'_j)\}_{j\in J'}$ also satisfies the above
conditions.

A \textit{branched projective structure} on $X$ with branching type $S$ is an equivalence 
class of data $\{(U_j,\, \phi_j)\}_{j\in J}$ satisfying the above conditions. This definition
was introduced in \cite{M1}, \cite{M2}.

We now give an equivalent geometric description in terms of a flat ${\mathbb C}{\mathbb P}^1$-bundle over $X$ and a
holomorphic section which fails to be transverse exactly at points in $S$ (compare with \cite{LM}).

Over ${\mathbb C}{\mathbb P}^1$, we have the trivial holomorphic bundle
$$
p_1\, :\, {\mathcal P}_0\, :=\, {\mathbb C}{\mathbb P}^1\times {\mathbb C}{\mathbb P}^1\, \longrightarrow {\mathbb C}{\mathbb P}^1
$$
with fiber ${\mathbb C}{\mathbb P}^1$, where $p_1$ is the projection to the first factor. This
projective bundle ${\mathcal P}_0$ is equipped with the trivial holomorphic connection, which we will denote by
$D_0$. The bundle ${\mathcal P}_0$ is also equipped with a holomorphic section
$$
s_0\, :\,{\mathbb C}{\mathbb P}^1 \, \longrightarrow\, {\mathbb C}{\mathbb P}^1\times {\mathbb C}{\mathbb P}^1
\,=\, {\mathcal P}_0\, ,\ \ x\, \longmapsto\, (x,\, x)\, .
$$

The earlier mentioned action of $\text{PGL}(2,{\mathbb C})$ on ${\mathbb C}{\mathbb P}^1$ 
lifts to ${\mathbb C}{\mathbb P}^1\times {\mathbb C}{\mathbb P}^1$ as the diagonal action. 
This action of $\text{PGL}(2,{\mathbb C})$ on ${\mathcal P}_0\,=\, {\mathbb C}{\mathbb 
P}^1\times {\mathbb C}{\mathbb P}^1$ evidently preserves the connection $D_0$ and also the 
above section $s_0$.

Given a branched projective structure $\{(U_j,\, \phi_j)\}_{j\in J}$, for every $j\, \in\, J$, we have
the projective bundle $\phi^*_j{\mathcal P}_0\, \longrightarrow\, U_j$ equipped with the
holomorphic connection $\phi^*_j D_0$ and the holomorphic section $\phi^*_j s_0$. Since
$({\mathcal P}_0,\, D_0,\, s_0)$ is $\text{PGL}(2,{\mathbb C})$--equivariant, for any
connected component $U\, \subset\,
U_j\bigcap U_{j'}$, the two bundles with connection and
section $(\phi^*_j{\mathcal P}_0,\, \phi^*_jD_0,\, \phi^*_j s_0)$ and
$(\phi^*_{j'}{\mathcal P}_0,\, \phi^*_{j'}D_0,\, \phi^*_{j'} s_0)$ patch together compatibly over
$U \subset U_j \bigcap U_{j'}$ using $f_{j,j'}\, \in\, \text{PGL}(2,{\mathbb C})$ (see the fourth
condition in the earlier definition of data giving a branched projective
structure). Therefore, we get a holomorphic projective bundle
\begin{equation}\label{e1}
q\, :\, {\mathcal P}_1\, \longrightarrow\, X
\end{equation}
equipped with a flat holomorphic connection $D_1$ and also a holomorphic section
\begin{equation}\label{s1}
s_1\,:\, X \, \longrightarrow\, {\mathcal P}_1
\end{equation}
(hence $q\circ s_1\,=\, \text{Id}_X$).

In the sequel, the holomorphic tangent bundle of a complex manifold $Z$ will be denoted by $TZ$.

Let $dq\, :\, T{\mathcal P}_1\, \longrightarrow\, q^*TX$ be the differential of the projection
$q$ in \eqref{e1}. The connection $D_1$ constructed above is given by a holomorphic homomorphism
$$
H_{D_1}\, :\, q^*TX\, \longrightarrow\, T{\mathcal P}_1
$$
such that $dq\circ H_{D_1}\,=\, \text{Id}_{q^*TX}$. The subbundle $H_{D_1}(q^*TX)\, \subset\,
T{\mathcal P}_1$ is called the horizontal subbundle for the connection $D_1$.
Let
$$
T_q\, :=\, \text{kernel}(dq)\, \subset\, T{\mathcal P}_1
$$
be the relative tangent bundle for the projection $q$. We note that
\begin{equation}\label{e2}
T{\mathcal P}_1\,=\, T_q\oplus H_{D_1}(q^*TX)\, .
\end{equation}
Let
\begin{equation}\label{e3}
\widehat{D}_1\, :\, T{\mathcal P}_1\, \longrightarrow\, T_q
\end{equation}
be the projection for the decomposition in \eqref{e2}.

Identify $X$ with the image $s_1(X)\, \subset\, {\mathcal P}_1$ using the map $s_1$ in \eqref{s1}.
The differential $ds_1$ identifies $TX$ with the tangent bundle $T(s_1(X))$ of the
manifold $s_1(X)$. Now consider the restriction
\begin{equation}\label{e4}
S(D_1)\, :=\, \widehat{D}_1\vert_{T(s_1(X))}\, :\, T(s_1(X))\,=\, TX\, \longrightarrow\, s^*_1 T_q\, ,
\end{equation}
where $\widehat{D}_1$ is the projection in \eqref{e3}. The divisor for the above section $S(D_1)$
of the line bundle $$\text{Hom}(TX,\, s^*_1 T_q)\,=\, (s^*_1 T_q)\otimes K_X
\,\longrightarrow\, X$$ is the divisor $S$ in
\eqref{e0}, where $K_X$ is the holomorphic cotangent bundle of $X$.

Conversely, let $q'_1\, :\, {\mathcal P}'\, \longrightarrow\, X$ be a holomorphic ${\mathbb C}{\mathbb 
P}^1$--bundle, equipped with a flat holomorphic connection $D'$ and also a holomorphic section $s'$. Let
$$
\widehat{D}'\, :\, T{\mathcal P}'\, \longrightarrow\, T_{q'_1}
$$
be the projection given by the connection $D'$ such that the kernel is the horizontal subbundle for $D'$
(as in \eqref{e3}). Then $({\mathcal P}',\, D',\, s')$ defines a
branched projective structure on
$X$ with branching type $S$ if and only if the divisor of the homomorphism restricted to $s'(X)$
$$
\widehat{D}'\vert_{T(s'(X))}\, :\, T(s'(X))\,=\, TX\, \longrightarrow\, (s')^* T_{q'_1}
$$
coincides with $S$. To see this construct local holomorphic trivializations of ${\mathcal P}'$
such that $D'$ becomes the trivial connection. Now use $s'$ to have local holomorphic
coordinate functions: they define the branched projective structure.

\subsection{Differential operators}\label{sdo}

For a holomorphic vector bundle $E$ on $X$ and any positive integer $n$, let $J^n(E)$ be
the $n$-th order jet bundle for $E$. (See 2(b) of \cite{BR}.) We recall that
$$
J^n(E)\,:=\, p_{1*}((p^*_2E)\otimes ({\mathcal O}_{X\times X}/{\mathcal O}_{X\times X}(-(n+1)
\cdot\Delta)))\, ,
$$
where $p_i\, :\, X\times X\, \longrightarrow\, X$ is projection to the $i$-th factor
($i\,=\, 1,\, 2$) and $\Delta\, \subset\, X\times X$ is the diagonal divisor.
There is a natural short exact sequence of vector
bundles
\begin{equation}\label{je1}
0\, \longrightarrow\, E\otimes K^{\otimes n}_X \,\stackrel{\iota_n}{\longrightarrow}\,
J^n(E)\, \longrightarrow\,J^{n-1}(E) \, \longrightarrow\, 0
\end{equation}
given by the inclusion of the sheaf ${\mathcal O}_{X\times X}(-(n+1)
\cdot\Delta)$ in ${\mathcal O}_{X\times X}(-n\cdot\Delta)$.
The sheaf of holomorphic differential operators of order $n$ from $E$ to a holomorphic vector bundle
$E'$ is defined to be
$$
\text{Diff}^n(E,\, E') \, :=\, \text{Hom}(J^n(E),\, E')\,=\, E'\otimes J^n(E)^*\, .
$$
Consider the homomorphism
\begin{equation}\label{symb}
\gamma_n\, :\, \text{Diff}^n(E,\, E')\,\longrightarrow\, \text{Hom}(E,\, E')\otimes (TX)^{\otimes n}
\end{equation}
defined by the composition
$$
\text{Diff}^n(E,\, E')\,=\, E'\otimes J^n(E)^*\, \stackrel{\text{Id}_{E'}\otimes \iota^*_n}{\longrightarrow}\,
E'\otimes E^*\otimes (TX)^{\otimes n}\,=\, \text{Hom}(E,\, E')\otimes (TX)^{\otimes n}\, ,
$$
where $\iota_n$ is the homomorphism in \eqref{je1}. This $\gamma_n$ is known as the symbol 
homomorphism; we shall use this terminology throughout the paper.

\section{Topological properties of projective bundles}\label{s3}

Henceforth, we will always assume that $X$ is compact.

The isomorphism classes of topological ${\mathbb C}{\mathbb P}^1$--bundles on $X$ are parametrized by ${\mathbb 
Z}/2\mathbb Z$. This classification can be described as follows. Given a ${\mathbb C}{\mathbb P}^1$--bundle
${\mathcal P}\, \longrightarrow\, X$, there is a complex vector bundle $E$ on $X$ of rank two such that
${\mathbb P}(E)\,=\, {\mathcal P}$. This $E$ is not unique. However, if $E'$ is another rank two
complex vector bundle on $X$ such that ${\mathbb P}(E')\,=\,{\mathbb P}(E)\,=\, {\mathcal P}$, then
$E'\,=\, E\otimes L$, where $L$ is a complex line bundle on $X$. This implies that
$\text{degree}(E')\,=\, \text{degree}(E)+2\cdot \text{degree}(L)$. Hence
$$
\text{degree}(E')\,\equiv \, \text{degree}(E)\ \ \text{ mod }\ 2\, .
$$
The isomorphism class of ${\mathcal P}$ is determined by the image of $\text{degree}(E)$ in ${\mathbb Z}/2{\mathbb Z}$.

Take a branched projective structure $P_1$ on $X$ with branching type $S$. Let
${\mathcal P}_1\, \longrightarrow\, X$ be the holomorphic ${\mathbb C}{\mathbb P}^1$--bundle corresponding to 
$P_1$ constructed in \eqref{e1}. Let $E$ be a holomorphic vector bundle on $X$ of rank two such that
${\mathbb P}(E)\,=\, {\mathcal P}_1$. Now the holomorphic section $s_1$ in \eqref{s1} produces a holomorphic
line subbundle
\begin{equation}\label{l1}
L_1\, \subset\, E\, .
\end{equation}
The holomorphic line bundle $s^*_1 T_q$ in \eqref{e4} is identified with $(E_1/L_1)\otimes L^*_1\,=\,
\text{Hom}(L_1,\, E/L_1)$. Since the divisor of the homomorphism $S(D_1)$ in \eqref{e4} is $S$, it
follows immediately that
\begin{equation}\label{is1}
s^*_1 T_q\,=\, \text{Hom}(L_1,\, E/L_1)\,=\, TX\otimes {\mathcal O}_X(S)\, .
\end{equation}
Therefore, we have
$$
\text{degree}(E)\, =\, \text{degree}(L_1)+\text{degree}(E/L_1)\,=\,
\text{degree}(\text{Hom}(L_1,\, E/L_1))+2\cdot \text{degree}(L_1)
$$
$$
=\, \text{degree}({\mathcal O}_X(S))+\text{degree}(TX)+2\cdot \text{degree}(L_1)
\,=\, \sum_{i=1}^d n_i +\text{degree}(TX)+2\cdot \text{degree}(L_1)\, .
$$
This implies that
$$
\text{degree}(E)\,\equiv \, \sum_{i=1}^d n_i\ \ \text{ mod }\ 2\, ,
$$
because $\text{degree}(TX)+2\cdot \text{degree}(L_1)$ is an even integer.

Henceforth, for simplicity, we will always assume the following:

\begin{assumption}\label{a1}\mbox{}
\begin{enumerate}
\item All $n_i\,=\, 1$, $1\,\leq\, i\, \leq\, d$, and

\item $d\,=\, \text{degree}({\mathcal O}_X(S))$ in an even integer.
\end{enumerate}
\end{assumption}

In view of Assumption \eqref{a1}(2), after substituting $E$ by $E\otimes L_2$, where
$L_2$ is a holomorphic line bundle on $X$ with $L^{\otimes 2}_2 \,=\, \bigwedge^2 E^*$, we
get that $$\bigwedge\nolimits^2 E\,=\, {\mathcal O}_X\, .$$
We will always use this normalization $\bigwedge\nolimits^2 E\,=\, {\mathcal O}_X$ of $E$, that is, the determinant line
bundle of $E$ will be trivial.

Notice that Assumption \ref{a1}(2)) is equivalent to the vanishing of the second Stiefel--Whitney class of the bundle $P_1$ constructed in \eqref{e1}. It is also equivalent to the fact that the monodromy representation
of the branched projective structure is liftable to $\text{SL}(2,{\mathbb C})$ (see Corollary 11.2.3 in \cite{GKM}).

Consider the holomorphic line subbundle $L_1$ in \eqref{l1}. From \eqref{is1}
we have
$$
\bigwedge\nolimits^2 E\,=\, L_1\otimes (E/L_1) \,=\, \text{Hom}(L_1, E/L_1)\otimes
L^{\otimes 2}_1\,=\, TX\otimes {\mathcal O}_X(S)\otimes L^{\otimes 2}_1\, .
$$
Since $\bigwedge^2 E\,=\, {\mathcal O}_X$, this implies that
\begin{equation}\label{e5b}
L^{\otimes 2}_1\,=\, K_X\otimes {\mathcal O}_X(-S)\,=\, ((E/L_1)^*)^{\otimes 2}\, .
\end{equation}
So we have
\begin{equation}\label{e5}
(E/L_1)^{\otimes 2}\,=\, TX\otimes {\mathcal O}_X(S)\, .
\end{equation}

Let $F$ be a holomorphic vector bundle on $X$ of rank two equipped with a holomorphic connection
$D$. Let $L\, \subset\, F$ be a holomorphic line subbundle. The quotient map
$F\, \longrightarrow\, F/L$ will be denoted by $q'$. The composition
$$
L\, \hookrightarrow\, F \, \stackrel{D}{\longrightarrow}\, F\otimes K_X
\, \stackrel{q'\otimes\text{Id}_{K_X}}{\longrightarrow}\, (F/L)\otimes K_X
$$
is evidently ${\mathcal O}_X$--linear. Let
\begin{equation}\label{sf}
{\mathcal F}(L, D)\, \in\, H^0(X,\, \text{Hom}(L,\, (F/L)\otimes K_X))
\end{equation}
be the homomorphism obtained by this composition; it is called the second fundamental form
of the subbundle $L$ for the connection $D$.

A holomorphic connection on $E$ induces a holomorphic connection on $\bigwedge^2 E\,=\, {\mathcal O}_X$.
Note that ${\mathcal O}_X$ has a canonical holomorphic connection given by the de Rham differential
$f\, \longmapsto\, df$.
This canonical holomorphic connection on ${\mathcal O}_X$ will be denoted by ${\mathcal D}_0$.

\begin{remark}\label{flat}
Note that any holomorphic connection on a Riemann surface is
automatically flat; the curvature vanishes because there are no nonzero $(2,0)$--forms.
\end{remark} 

\begin{lemma}\label{lem1}
Giving a branched projective structure on $X$, with branching type $S$, is equivalent to
giving a triple $(F,\, L,\, D)$, where
\begin{enumerate}
\item $F$ is a holomorphic vector bundle on $X$ of rank two with $\bigwedge^2 F\,=\, {\mathcal O}_X$,

\item $L\, \subset\, F$ is a holomorphic line subbundle whose degree is ${\rm genus}(X)-1-\frac{d}{2}$,

\item $D$ is a holomorphic connection on $F$ such that the divisor for the section
${\mathcal F}(L, D)$ in \eqref{sf} is $S$, and

\item the holomorphic connection on $\bigwedge^2 F\,=\, {\mathcal O}_X$ induced by $D$ coincides with the
canonical connection ${\mathcal D}_0$.
\end{enumerate}
\end{lemma}

\begin{proof}
Take a branched projective structure $P_1$ on $X$ with branching type $S$.
We saw that $P_1$ gives
\begin{enumerate}
\item a holomorphic vector bundle $F$ on $X$ of rank two with $\bigwedge^2 F\,=\, {\mathcal O}_X$,

\item a holomorphic line subbundle $L\, \subset\, F$ whose degree is ${\rm genus}(X)-1-\frac{d}{2}$
(see \eqref{e5b}), and

\item a holomorphic connection $D_1$ on ${\mathcal P}_1\,=\, {\mathbb P}(F)$ such that the
divisor for the homomorphism $S(D_1)$ constructed as in \eqref{e4} is $S$.
\end{enumerate}

Since $\bigwedge^2 F\,=\, {\mathcal O}_X$, giving a holomorphic connection on $F$,
such that the connection on $\bigwedge^2 F\,=\, {\mathcal O}_X$ induced by it coincides with the
canonical connection ${\mathcal D}_0$, is equivalent to giving a holomorphic connection on
the projective bundle ${\mathbb P}(F)$. Indeed, this follows immediately from the
fact that the homomorphism of Lie algebras corresponding to the quotient homomorphism of Lie groups
$$
\text{SL}(2, {\mathbb C})\, \longrightarrow\, \text{PGL}(2, {\mathbb C})
$$
is an isomorphism. A holomorphic connection is a Lie algebra valued holomorphic one-form satisfying
certain conditions on the total space of the corresponding principal bundle. The principal 
$\text{PGL}(2, {\mathbb C})$--bundle for ${\mathbb P}(F)$ is the quotient of the principal
$\text{SL}(2, {\mathbb C})$--bundle for $F$ by the action of the center of $\text{SL}(2, {\mathbb C})$. So there
is an isomorphism between the connections on them simply by pulling back the connection form.

Let $D$ be the holomorphic connection on $F$ corresponding to the
above mentioned holomorphic connection $D_1$ on ${\mathbb P}(F)$.

Recall that the line subbundle $L\, \subset\, F$ corresponds to the section
$s_1$ of ${\mathcal P}_1$ in \eqref{e4} by the identification ${\mathcal P}_1\,=\, {\mathbb P}(F)$.
Using the isomorphism $s^*_1 T_q\,=\, \text{Hom}(L,\, F/L)$ (given in \eqref{is1}), the section
${\mathcal F}(L, D)$ in \eqref{sf} corresponds to the section $S(D_1)$ constructed as in \eqref{e4}.
Therefore, the triple $(F,\, L,\, D)$ satisfies all the conditions in the statement of the lemma.

Conversely, take $(F,\, L,\, D)$ satisfying the conditions in the lemma. Then $L$ defines a holomorphic
section of the holomorphic projective bundle ${\mathbb P}(F)$; this section
will be denoted by $s_1$. The holomorphic connection $D$ on $F$ induces a holomorphic connection on
${\mathbb P}(F)$; this induced connection on ${\mathbb P}(F)$ will be denoted by $D_1$.

Since the section ${\mathcal F}(L, D)$ in \eqref{sf} coincides with the one constructed as
in \eqref{e4}, the triple $({\mathbb P}(F), \, s_1,\, D_1)$ produces a 
branched projective structure on $X$ with branching type $S$.
\end{proof}

\begin{remark}\label{rem1}
Fix a holomorphic line bundle $Q$ on $X$ such that $Q^{\otimes 2}\,=\, TX\otimes {\mathcal O}_X(S)$
(as in \eqref{e5}). In Lemma \ref{lem1} we may choose $F$ such that $F/L\,=\, Q$. To see
this take any triple $(F,\, L,\, D)$ satisfying the conditions in Lemma \ref{lem1}. Then
$(F/L)\otimes L_0\, =\, Q$, where $L_0$ is a holomorphic line bundle of order two, because
$(F/L)^{\otimes 2}\,=\, TX\otimes {\mathcal O}_X(S)\,=\, Q^{\otimes 2}$
(see \eqref{e5}). There is a unique holomorphic connection
$D_0$ on $L_0$ such that the connection on $L^{\otimes 2}\,=\, {\mathcal O}_X$ induced by $D_0$
coincides with the trivial connection ${\mathcal D}_0$ on ${\mathcal O}_X$ given by the de Rham differential.
Let $D'$ be the holomorphic connection on $F\otimes L_0$ induced by $D$ and $D_0$. Then the
triple $(F\otimes L_0,\, L\otimes L_0,\, D')$ satisfies all the conditions in Lemma \ref{lem1}.
Note that using the natural isomorphism $$\text{Hom}(L,\, (F/L)\otimes K_X)\,=\, (F/L)\otimes K_X\otimes L^*
\,=\, \text{Hom}(L\otimes L_0,\, ((F\otimes L_0)/(L\otimes L_0))\otimes K_X)\, ,$$
the second fundamental form ${\mathcal F}(L, D)$ in \eqref{sf} coincides with the second fundamental
${\mathcal F}(L\otimes L_0, D')$ of the subbundle $L\otimes L_0$ for the
connection $D'$. The branched projective structure on $X$, with branching type $S$,
given by the triple $(F\otimes L_0,\, L\otimes L_0,\, D')$ clearly coincides with the
one given by $(F,\, L,\, D)$.
\end{remark}

\section{Logarithmic connection and residue}\label{sec.l}

In this section $Y$ is any connected Riemann surface and $S'$ is a finite subset of points of $Y$.
Like before, the divisor
on $Y$ given by the formal sum of the points of $S'$ will also be denoted by $S'$.
The holomorphic cotangent bundle of $Y$ will be denoted by $K_Y$.

We note that for any point $y\, \in\, S'$, the
fiber $(K_Y\otimes {\mathcal O}_Y(S'))_y$ is identified with $\mathbb C$ by sending any
meromorphic $1$-form defined around
$y$ to its residue at $y$. More precisely, for any holomorphic coordinate function
$z$ on $Y$ defined around the point $y$ with $z(y)\,=\, 0$, consider the homomorphism
\begin{equation}\label{ry}
R_y\, :\, (K_Y\otimes {\mathcal O}_Y(S'))_y\, \longrightarrow\, {\mathbb C}\, ,
\ \ c\cdot \frac{dz}{z} \, \longmapsto\, c\, .
\end{equation}
This homomorphism is in fact independent of the choice of the above coordinate function $z$.

Let $V$ be a holomorphic vector bundle on $Y$.
A \textit{logarithmic connection} on $V$ singular over $S'$ is a holomorphic
differential operator of order one
$$
D\, :\, V\, \longrightarrow\, V\otimes K_Y\otimes {\mathcal O}_Y(S')
$$
such that $D(fs) \,=\, f D(s) + s\otimes df$ for all locally defined holomorphic function $f$ and
all locally defined holomorphic section $s$ of $V$. In other words,
$$
D\, \in\, H^0(Y,\,\text{Hom}(J^1(V),\, V\otimes K_Y\otimes {\mathcal O}_Y(S')))\,=\,
H^0(Y,\,\text{Diff}^1(V,\, V\otimes K_Y\otimes {\mathcal O}_Y(S')))
$$
such that the symbol of $D$
is the holomorphic section of $\text{End}(V)\otimes {\mathcal O}_Y(S')$ given by $\text{Id}_V\otimes 1$.

For a logarithmic connection $D$ on $V$ singular over $S'$, and a point $y\, \in\, S'$, consider
the composition
$$
V\, \stackrel{D}{\longrightarrow}\, V\otimes K_Y\otimes {\mathcal O}_Y(S')
\, \stackrel{\text{Id}_V\otimes
R_y}{\longrightarrow}\, V_y\otimes{\mathbb C}\,=\, V_y\, ,
$$
where $R_y$ is the residue homomorphism constructed in \eqref{ry}. This composition homomorphism vanishes
on the subsheaf $V\otimes {\mathcal O}_Y(-y)\, \subset\, V$, and hence it produces a homomorphism
$$
\text{Res}(D,y)\, :\, V/(V\otimes {\mathcal O}_Y(-y))\,=\, V_y\, \longrightarrow\, V_y\, .
$$
This endomorphism $\text{Res}(D,y)$ of $V_y$ is called the \textit{residue} of the
connection $D$ at the point $y$; see \cite[p.~53]{De}.

Fix a point $y\, \in\, S'$. Let $L$ be a holomorphic line bundle
on $Y$, and let
$$
D\, :\, L\, \longrightarrow\, L\otimes K_Y\otimes {\mathcal O}_Y(S')
$$
be a logarithmic connection on $L$ singular over $S'$ such that the residue
$$
\text{Res}(D,y)\, =\, - 1 \,=\, -\text{Id}_{L_y}\, .
$$

\begin{lemma}\label{lemn1}
Let $s$ be a holomorphic section of $L$ defined on an open neighborhood $U$ of the point
$y$ of $S'$ such that $s(y)\,=\, 0$.
Then the section $$D(s)\, \in\, H^0(U,\, (L\otimes K_Y\otimes {\mathcal O}_Y(S'))\vert_U)$$
vanishes at $y$ at order at least two.
\end{lemma}

\begin{proof}
The local model of $(L,\, D)$ around $y$ is the line bundle ${\mathcal O}_Y(y)$ equipped with the
logarithmic connection $D_0$ given by the de Rham differential. For a holomorphic coordinate function $z$ on
$U$ with $z(y)\,=\, 0$, and a holomorphic function $f$ defined on $U$,
we have
$$
\frac{df}{dz}\,=\, z^2 (\frac{f'}{z})(\frac{dz}{z})\, .
$$
Hence the lemma follows.
\end{proof}

Let $E$ be a holomorphic vector bundle on $Y$ of rank two. Fix two distinct lines $\ell_0,\, \ell_1\, \subset\, E_y$
(so $E_y\,=\, \ell_0\oplus \ell_1$), where $y\, \in\, S'$ is a fixed point as before. Let
$$
D\,:\, E\, \longrightarrow\, E\otimes K_Y\otimes {\mathcal O}_Y(S')
$$
be a logarithmic connection on $E$ singular over $S'$, such that
\begin{itemize}
\item the residue $\text{Res}(D,y)$ has two eigen-values $-1$ and $0$, and

\item $\ell_1$ (respectively, $\ell_0$) is the eigen-line for the eigen-value $-1$ (respectively, $0$)
of $\text{Res}(D,y)$.
\end{itemize}

\begin{proposition}\label{propn1}\mbox{}
\begin{enumerate}
\item The logarithmic $D$ connection produces a homomorphism
$$
\rho(D, y) \, :\, \ell_0\, \longrightarrow\, \ell_1\otimes (K_Y)_y\, ,
$$
where $(K_Y)_y$ is the fiber of $K_Y$ over the point $y\, \in\, Y$.

\item The local monodromy of $D$ around the point $y$ is trivial if and only if $\rho(D, y) \,=\, 0$.
\end{enumerate}
\end{proposition}

\begin{proof}
To construct the homomorphism $\rho(D, y)$, take any vector $\alpha\, \in\, \ell_0$. Let
$s_\alpha$ be a holomorphic section of $E$, defined on an open neighborhood of $y$, such that $s_\alpha(y)\,=\, \alpha$.
Now consider the locally defined holomorphic section $D(s_\alpha)$ of the vector
bundle $E\otimes K_Y\otimes {\mathcal O}_Y(S')$. Since $\ell_0$ is the eigen-bundle of
the residue $\text{Res}(D,y)$ for the eigen-value zero, it follows that the evaluation
$D(s_\alpha)(y)$, of the section $D(s_\alpha)$ at $y$, vanishes. This implies that
$\text{Res}(D,y)$ is given by a locally defined holomorphic section of the vector bundle
$E\otimes K_Y$ using the natural inclusion of the sheaf $E\otimes K_Y$ in $E\otimes K_Y\otimes {\mathcal O}_Y(S')$.
This section of $E\otimes K_Y$ giving $\text{Res}(D,y)$ will also be denoted
by $\text{Res}(D,y)$. So the evaluation $D(s_\alpha)(y)$ of this section of $E\otimes K_Y$ is an element of
\begin{equation}\label{nds}
(E\otimes K_Y)_y\,=\, (\ell_0\otimes (K_Y)_y)\oplus (\ell_1\otimes (K_Y)_y)\, .
\end{equation}
Let $\widetilde{\alpha}$ be the component of $D(s_\alpha)(y)\, \in\, (E\otimes K_Y)_y$ in the direct summand
$\ell_1\otimes (K_Y)_y$ in \eqref{nds}

We will show that the above element $$\widetilde{\alpha}\,\in\, \ell_1\otimes (K_Y)_y$$ is independent
of the choice of the holomorphic section $s_\alpha$ passing through $\alpha$.

To prove the above independence, first set $\alpha\,=\, 0$. Then from the computation in the proof of
Lemma \ref{lemn1} if follows that $\widetilde{\alpha}\,=\, 0$. From this it follows that for any
general $\alpha$ (not necessarily the zero vector), the element $\widetilde{\alpha}$ is independent
of the choice of the holomorphic section $s_\alpha$ passing through $\alpha$.

Now define
$$
\rho(D, y) \, :\, \ell_0\, \longrightarrow\, \ell_1\otimes (K_Y)_y\, , \ \ \alpha\, \longmapsto\,
\widetilde{\alpha}\, .
$$
We have shown that this map is well-defined.

To prove the second statement of the proposition,
let $F$ be the holomorphic vector bundle on $Y$ of rank two that fits in the following
short exact sequence of sheaves on $Y$
\begin{equation}\label{gn2}
0\, \longrightarrow\, F\, \stackrel{\iota''}{\longrightarrow}\, E\,
\longrightarrow\, E_y/\ell_0\, \longrightarrow\, 0\, .
\end{equation}
Since $\ell_0$ is an eigen-space for ${\rm Res}(D, y)$, the logarithmic connection $D$
on $E$ induces a logarithmic connection on the subsheaf $F$ in \eqref{gn2}. Let $D'$ denote
the logarithmic connection on $F$ induced by $D$.

Consider the homomorphism of
fibers $F_y \, \longrightarrow\, E_y$ given by the homomorphism $\iota''$ of sheaves in
\eqref{gn2}. Since the image of $F_y$ under this homomorphism is contained in the eigen-space of
${\rm Res}(D, y)$ for the eigen-value zero (in fact $\iota''(F_y)\,\subset\, E_y$
is the eigen-space for the eigen-value zero of ${\rm Res}(D, y)$), it follows that the
residue of the logarithmic connection $D'$ at $y$ has only zero as the eigen-value. Indeed, this
is an immediate consequence of the
fact that the eigen-values of the residue ${\rm Res}(D, y)$ give the
eigen-values of the residue ${\rm Res}(D', y)$ of the induced connection $D'$.
Therefore, the residue ${\rm Res}(D', y)$ is nilpotent.

It is now straight-forward to check that ${\rm Res}(D', y)$ is given by the homomorphism
$\rho(D, y)$ in the first statement of the proposition. Hence ${\rm Res}(D', y)\,=\, 0$ if and
only if we have $\rho(D, y)\,=\, 0$.

The local monodromy of $D$ around $y$ evidently coincides with the local monodromy of $D'$ around $y$, because
the two vector bundles with connection, namely, $(E,\, D)$ and $(F,\, D')$, are canonically identified, using
$\iota''$ in \eqref{gn2}, over the complement of the point $y\, \in\, Y$.
On the other hand, the local monodromy of $D'$ around $y$
is trivial if and only if the nilpotent residue ${\rm Res}(D', y)$ actually vanishes (see Remark \ref{rem-res}). This completes the
proof of the proposition.
\end{proof}

\begin{remark}\label{rem-res}
If $D$ is a logarithmic connection on $E$ singular over $y$, then the eigen-values of the local monodromy
of $D$ around the point $y$ are of the form $\exp(2\pi\sqrt{-1}\mu_1),\, \exp(2\pi\sqrt{-1}\mu_2)$, where $\mu_1$ and
$\mu_2$ are the eigen-values of ${\rm Res}(D, y)$. Therefore, if ${\rm Res}(D, y)$ is nonzero nilpotent, then the
local monodromy of $D$ around the point $y$ is conjugate to the matrix
$$
\begin{pmatrix}
1 & 1\\
0 & 1
\end{pmatrix} \, \in\, {\rm SL}(2, {\mathbb C})\, .
$$
\end{remark}

\section{Branched projective structure as logarithmic connection on jet bundle}\label{s4}

Assume that $d\,:=\, \# S\, \not=\, 2g-2$ (see also Assumption \ref{a1}(2)).

Let $Q$ be a holomorphic line bundle on $X$ such that
\begin{equation}\label{dq}
Q^{\otimes 2}\,=\, TX\otimes {\mathcal O}_X(S)
\end{equation}
(same condition as in \eqref{e5}). In particular, we have
\begin{equation}\label{dq2}
\text{degree}(Q)\,=\, \frac{d}{2} -g +1 \, \not=\, 0
\end{equation}
because $d\, \not=\, 2g-2$.

Let $J^1(Q)$ denote the first order jet bundle for $Q$. It fits in the following short
exact sequence of vector bundles on $X$:
\begin{equation}\label{e6}
0\, \longrightarrow\, Q\otimes K_X\, \stackrel{\iota_0}{\longrightarrow}\, J^1(Q)\,
\stackrel{q_0}{\longrightarrow}\, J^0(Q)\,=\, Q\, \longrightarrow\, 0
\end{equation}
(see \eqref{je1}).

For notational convenience, we will often identify $\iota_0(Q\otimes K_X)$ with $Q\otimes K_X$ using $\iota_0$.

From \eqref{dq} and \eqref{e6} it follows that
\begin{equation}\label{det}
\bigwedge\nolimits^2 J^1(Q)\,=\, Q\otimes K_X\otimes Q \,=\, {\mathcal O}_X(S)\, .
\end{equation}

We note that ${\mathcal O}_X(S)$ has a canonical logarithmic connection given by the
de Rham differential. Indeed, the sheaf of sections of ${\mathcal O}_X(S)$ are locally
defined meromorphic functions with pole of order at most one on $S$. For any such function
$f$, the differential $df$ has pole on $S$ of order at most two. 

The canonical logarithmic connection on ${\mathcal O}_X(S)$ given by the
de Rham differential will
be denoted by ${\mathcal D}_S$. The singular locus of ${\mathcal D}_S$ is $S$, and the
residue of ${\mathcal D}_S$ at any $y\, \in\, S$ is $-1$.

Any logarithmic connection on the vector bundle $J^1(Q)$
induces a logarithmic connection on $\bigwedge\nolimits^2 J^1(Q)\,=\, {\mathcal O}_X(S)$
(this isomorphism is in \eqref{det}).

\begin{theorem}\label{thm1}
Giving a branched projective structure on $X$ with branching type $S$ is equivalent to
giving a flat logarithmic connection $D^1$ on $J^1(Q)$ such that
\begin{itemize}
\item $D^1$ is nonsingular over $X\setminus S$,

\item ${\rm trace}({\rm Res}(D^1, x_i))\, =\, -1$ and ${\rm Res}(D^1,x_i)(w)\,=\, -w$
for every $x_i\, \in \, S$ and $w\, \in\, (Q\otimes K_X)_{x_i}$, where
$(Q\otimes K_X)_{x_i}\, \subset\, J^1(Q)_{x_i}$ is the line in \eqref{e6},

\item for every $y\,\in\, S$, the homomorphism $\rho(D^1, y)$ in Proposition \ref{propn1} vanishes, and

\item the logarithmic connection on $\bigwedge\nolimits^2 J^1(Q)\,=\, {\mathcal O}_X(S)$ induced by
$D^1$ coincides with the canonical logarithmic connection ${\mathcal D}_S$.
\end{itemize}
\end{theorem}

\begin{proof}
Let $P_1$ be a branched projective structure on $X$ with branching type $S$. In view of Lemma
\ref{lem1} and Remark \ref{rem1}, this $P_1$ gives a triple $(F,\, L,\, D)$ satisfying the
conditions in Lemma \ref{lem1} such that $F/L\, =\, Q$. Let
\begin{equation}\label{qp}
q'\, :\, F\, \longrightarrow\, Q\,=\, F/L
\end{equation}
be the quotient homomorphism.

Using the flatness of $D$ (see Remark \ref{flat}) we will construct a homomorphism
\begin{equation}\label{vp}
\varphi\, :\, F\, \longrightarrow\, J^1(Q)\, .
\end{equation}
For this, take any $x\, \in\, X$ and $v\, \in\, F_x$. Let $\widehat{v}$ be the unique flat section 
of $(F,\, D)$, defined on a simply connected neighborhood of $x$, such that $\widehat{v}(x)\,=\, 
v$. Therefore, $q'(\widehat{v})$ is a holomorphic section of $Q$ defined around $x$, where $q'$ 
is the projection in \eqref{qp}. Now restricting $q'(\widehat{v})$ to the first order 
neighborhood of $x$ we get an element $q'(\widehat{v})'\,\in\, J^1(Q)_x$. The homomorphism 
$\varphi$ in \eqref{vp} is defined by
$$
\varphi(v)\, =\, q'(\widehat{v})'\, .
$$

We will show that $\varphi$ constructed in \eqref{vp} is an isomorphism over the complement 
$X\setminus S$.

For this, take any $x\, \in\, X$ and
\begin{equation}\label{v}
v\, \in\, F_x
\end{equation}
such that $\varphi (v)\,=\, 0$.
Now consider the commutative diagram
\begin{equation}\label{cd}
\begin{matrix}
F & \stackrel{\varphi}{\longrightarrow} & J^1(Q)\\
~ \Big\downarrow q' && ~ \Big\downarrow q_0\\
Q & \stackrel{=}{\longrightarrow} & Q
\end{matrix}
\end{equation}
where $q'$ and $q_0$ are the homomorphisms in \eqref{qp} and \eqref{e6} respectively. Note that
both $q'$ and $q_0$ are surjective. Therefore, we conclude that $v$ in \eqref{v} satisfies the equation
$$
v\, \in\, \text{kernel}(q'(x))\,=\, L_x\, \subset\, F_x\, .
$$

{}From the commutative diagram in \eqref{cd} it follows that
$$
\varphi(L)\, \subset\, \text{kernel}(q_0)\,=\, Q\otimes K_X\, .
$$
{}From the construction of the second fundamental form ${\mathcal F}(L, D)$ in \eqref{sf}
it follows immediately that ${\mathcal F}(L, D)$ coincides with
$$
\varphi\vert_L\, :\, L\, \longrightarrow\, Q\otimes K_X\,=\, (F/L)\otimes K_X\, .
$$
On the other hand, the divisor for ${\mathcal F}(L, D)$ is $S$, so ${\mathcal F}(L, D)$
does not vanish on the complement $X\setminus S$. Therefore, we conclude that the element $x$ in
\eqref{v} lies in $S$, and $v\, \in\, L_x$. Conversely, for any $x\,\in\, S$, we have
$\varphi(x)(L_x)\,=\, 0$, because ${\mathcal F}(L, D) (x) \,=\,0$ and ${\mathcal F}(L, D) (x)\,=\,
\varphi\vert_L (x)$.

For each $y\, \in\, S$, let $$\ell^0_y\,:=\, \varphi(y)(F_y)\,=\,
\text{image}(\varphi(y))\, \subset\, J^1(Q)_y$$
be the line. From the commutative diagram in \eqref{cd} it follows that $q_0(y)(\ell^0_y)\,=\,
Q_y$. Hence we have a direct sum decomposition
\begin{equation}\label{d2}
J^1(Q)_y\,=\, \ell^0_y\oplus (Q\otimes K_X)_y\,=\, \ell^0_y\oplus \text{kernel}(q_0)_y
\end{equation}
of the fiber $J^1(Q)_y$.

The vector bundle $F$ and $J^1(Q)$ are evidently related by the following short exact sequence
of sheaves:
\begin{equation}\label{e12}
0\, \longrightarrow\, F\, \stackrel{\varphi}{\longrightarrow}\, J^1(Q) \, \longrightarrow\,
\bigoplus_{y\in S} J^1(Q)_y/\ell^0_y \, \longrightarrow\,0\, .
\end{equation}
Note that from \eqref{det} we have
$$
\text{degree}(J^1(Q)) - \text{degree}(F)\,=\, \text{degree}(J^1(Q))\,=\, 
\#S \,=\, d\, ;
$$
(recall from Lemma \ref{lem1}(1) that $\bigwedge^2 F\,=\, {\mathcal O}_X$).
Hence $\text{degree}(J^1(Q)/\varphi(F))\,=\, d$, which also follows from \eqref{e12}.

The holomorphic connection $D$ on the subsheaf $\varphi(F)\, \subset\, J^1(Q)$ extends to a 
logarithmic connection on $J^1(Q)$. Indeed, $D$ induces a logarithmic connection on $F\otimes 
{\mathcal O}_X(S)$, which we denote by $D^2$. The singular locus of $D^2$ is $S$, and for each 
$y\, \in\, S$, the residue of $D^2$ at $y$ is $-\text{Id}_{F_y}$. Indeed, this follows immediately
from the fact that the residue of the logarithmic connection on ${\mathcal O}_X(S)$ given by the
de Rham differential is $-1$ at each point of $S$. On the other hand, from 
\eqref{e12} it follows, by taking the duals, that
$$
F^*\, \supset\, J^1(Q)^*\, \supset \, F^*\otimes {\mathcal O}_X(-S)\, ,
$$
because $J^1(Q)\otimes {\mathcal O}_X(-S)\, \subset\, F$. Now again taking dual, we have
\begin{equation}\label{in2}
J^1(Q)\, \subset\, F\otimes {\mathcal O}_X(S)\, .
\end{equation}
The above logarithmic 
connection $D^2$ on $F\otimes {\mathcal O}_X(S)$ preserves the subsheaf $J^1(Q)$
in \eqref{in2}. Indeed, this follows from the fact that the residue of $D^2$ at any
$y\,\in\, S$ is $-\text{Id}_{F_y}$. Hence $D^2$ 
induces a logarithmic connection on $J^1(Q)$. The logarithmic connection on $J^1(Q)$ induced by 
$D^2$ will be denoted by $D^1$. The singular locus of $D^1$ coincides with that of
$D^2$, namely the subset $S$.

{}From the earlier observation that $\text{Res}(D^2,y)\,=\, -\text{Id}_{F_y}$ for $y\, \in\, S$,
and the above construction of $D^1$, it follows that the residue
$${\rm Res}(D^1, y)\, \in\, \text{End}(J^1(Q)_y)$$ of $D^1$ at $y$ preserves
the decomposition in \eqref{d2} for every $y\, \in\, S$. Moreover, we have
$$
{\rm Res}(D^1, y)(w_1,\, w_2)\,=\, - w_2\, ,\ \ \forall\ 
w_1 \,\in\, \ell^0_y\, ,\ w_2\,\in\, (Q\otimes K_X)_y
$$
for each $y\, \in\, S$. Note that the image of $\ell^0_y$ in the fiber
$(F\otimes {\mathcal O}_X(S))_y$ by the inclusion map of sheaves in \eqref{in2} vanishes
(this is because the image of $F_y$ in $(F\otimes {\mathcal O}_X(S))_y$ vanishes).

Since $D$ is a regular holomorphic connection, it does not have any nontrivial local monodromy.
Therefore, for any point $y\, \in\, S$, the local monodromy of $D^1$ around $y$ is trivial. Indeed,
the two vector bundles with holomorphic connections, namely $(F,\, D)$ and $(J^1(Q),\, D^1)$, are canonically
identified over the complement $X\setminus S$
using $\varphi$ in \eqref{e12}. Hence from the second statement in Proposition \ref{propn1}
it follows that the homomorphism $\rho(D^1, y)$ vanishes for all $y\, \in\, S$.

From the exact sequence in \eqref{e12} it follows immediately that the homomorphism of
second exterior powers induced by $\varphi$
$$
\bigwedge\nolimits^2\varphi\, :\, 
{\mathcal O}_X\,=\, \bigwedge\nolimits^2 F\, \longrightarrow\, \bigwedge\nolimits^2 J^1(Q)\,=\,
{\mathcal O}_X(S)
$$
(see \eqref{det} for the last isomorphism) coincides with the natural inclusion of
${\mathcal O}_X$ in ${\mathcal O}_X(S)$. Since $D$ induces the canonical connection ${\mathcal D}_0$
on $\bigwedge\nolimits^2 F\,=\, {\mathcal O}_X$ given by the de Rham differential (see Lemma \ref{lem1}(4)), the
logarithmic connection on $\bigwedge\nolimits^2 J^1(Q)\,=\,
{\mathcal O}_X(S)$ induced by $D^1$ also coincides with the one given by the de Rham
differential (in other words, it is
the logarithmic connection ${\mathcal D}_S$). Therefore, $D^1$ satisfies all the
three conditions in the statement of the theorem.

To prove the converse, let $D^1$ be a logarithmic connection on the vector bundle $J^1(Q)$, singular over
$S$, such that for every $y\, \in\, S$, we have
$${\rm trace}({\rm Res}(D^1, y))\, =\, -1\ \ \text{ and }\ \ {\rm Res}(D^1,y)(w)\,=\, -w$$
for all $w\, \in\, (Q\otimes K_X)_y$, the logarithmic connection on $\bigwedge^2 J^1(Q)
\,=\, {\mathcal O}_X(S)$ induced by $D^1$ coincides with the tautological logarithmic connection
${\mathcal D}_S$, and the homomorphism $\rho(D^1, y)$ in Proposition \ref{propn1} vanishes
for every $y\,\in\, S$.

Therefore, the eigen-values of ${\rm Res}(D^1, y)$ are $-1$ and $0$ for all $y\, \in\, S$. Denote
$$
\ell'_y\, :=\, \text{kernel}({\rm Res}(D^1, y))\, \subset\, J^1(Q)_y\, ,
$$
the eigen-space for $0$. As done in \eqref{gn2},
let $F'$ be the holomorphic vector bundle on $X$ of rank two that fits in the following
short exact sequence of sheaves on $X$
\begin{equation}\label{g2}
0\, \longrightarrow\, F'\, \stackrel{\iota}{\longrightarrow}\, J^1(Q)\,
\longrightarrow\, \bigoplus_{y\in S} J^1(Q)_y/\ell'_y\, \longrightarrow\, 0\, .
\end{equation}
From \eqref{g2} it follows immediately that
\begin{equation}\label{dz}
\bigwedge\nolimits^2 F'\,=\, (\bigwedge\nolimits^2 J^1(Q))\otimes {\mathcal O}_X(-S)\,=\,
{\mathcal O}_X(S)\otimes {\mathcal O}_X(-S)\,=\, {\mathcal O}_X
\end{equation}
(see \eqref{det}).

Repeating the argument in the 
proof of the second statement of Proposition \ref{propn1} we conclude that the connection
$D'$ on the vector bundle $F'$ induced from $D^1$ is in fact a regular holomorphic connection.

Since $\iota$ in \eqref{g2} is an isomorphism over $X\setminus S$, the line subbundle
$Q\otimes K_X$ of $J^1(Q)$ in \eqref{e6} produces a line subbundle of $F'\vert_{X\setminus S}$.
Let
\begin{equation}\label{lp}
L'\, \subset\, F'
\end{equation}
be the line subbundle over $X$ generated by this
line subbundle of $F'\vert_{X\setminus S}$. We note that this $L'$ is uniquely determined by the
condition that a locally defined holomorphic section $s$ of $F'$ is a section of $L'$ if and only
if $\iota(s)$ is a section of the subbundle $Q\otimes K_X$ of
$J^1(Q)$. From this it follows that $\iota$ induces a homomorphism
\begin{equation}\label{itp}
\iota'\, :\, F'/L'\, \longrightarrow\, J^1(Q)/(Q\otimes K_X)\,=\, Q\, .
\end{equation}
For each point $y\, \in\, S$, the composition
$$
\iota'(F'_y)\,=\, \ell'_y \, \hookrightarrow\, J^1(Q)_y\, \longrightarrow\,
J^1(Q)_y/(Q\otimes K_X)_y\,=\, Q_y
$$
is surjective, and hence from the exact sequence in \eqref{g2} it follows that the
homomorphism $\iota'$ in \eqref{itp} is an isomorphism. Consequently,
we have
\begin{equation}\label{it2}
-\text{degree}(L')\,=\, \text{degree}(F'/L') \,=\, \text{degree}(Q)\,=\,
\frac{d}{2} -g +1
\end{equation}
(see \eqref{dq2} and \eqref{dz} for the above equalities).

\medskip
\textbf{Claim\, A.}\, We will show that the logarithmic connection $D^1$ on $J^1(Q)$ does not preserve
the subbundle $Q\otimes K_X$ in \eqref{e6}.
\medskip

To prove the claim by contradiction, assume that $D^1$ preserves $Q\otimes K_X$. Then $D^1$
induces a logarithmic connection on the quotient bundle $J^1(Q)/(Q\otimes K_X)\,=\, Q$ in
\eqref{e6}. Let $\widehat{D}^1$ be this induced logarithmic connection on $Q$. The
residue of $\widehat{D}^1$ at any $x_i\, \in\, S$ is induced by the residue of $D^1$ at $x_i$. Note that
for each $x_i\, \in\, S$, the endomorphism of $Q_{x_i}$ induced by the residue 
${\rm Res}(D^1, x_i)\,\in\, {\rm End}(J^1(Q)_{x_i})$ is the zero map. Consequently, all the residues of the logarithmic connection
$\widehat{D}^1$ vanish. Therefore, $\widehat{D}^1$ is in fact a regular holomorphic connection on $Q$. This
implies that we have $$\text{degree}(Q)\,=\, 0\, ,$$ because holomorphic
connections on $X$ are flat (Remark \ref{flat}). But this contradicts \eqref{dq2}. Hence
we conclude that $D^1$ does not preserve the subbundle $Q\otimes K_X\,\subset\,
J^1(Q)$. This proves Claim A.

Let
\begin{equation}\label{deta}
\eta\, :\, Q\otimes K_X\, \longrightarrow\, Q\otimes K_X\otimes {\mathcal O}_X(S)
\end{equation}
be the homomorphism given by the composition
$$
Q\otimes K_X\, \stackrel{\iota_0}{\hookrightarrow}\, J^1(Q)\, \stackrel{D^1}{\longrightarrow}\,
J^1(Q)\otimes K_X\otimes {\mathcal O}_X(S) \, \stackrel{q_0\otimes {\rm Id}}{\longrightarrow}\, 
Q\otimes K_X\otimes {\mathcal O}_X(S)\, ,
$$
where $\iota_0$ and $q_0$ are the homomorphisms in \eqref{e6},
while ${\rm Id}$ stands for the identity map of the line bundle
$K_X\otimes {\mathcal O}_X(S)$. From the Leibniz identity for $D^1$ it follows immediately that
this $\eta$ is in fact ${\mathcal O}_X$--linear. Therefore, $\eta$ is a holomorphic section of ${\mathcal
O}_X(S)$. Note that this is the second fundamental form of the subbundle $Q\otimes K_X$ for the
logarithmic connection $D^1$.

Since $D^1$ does not preserve the subbundle $Q\otimes K_X$ (Claim A),
the above section $\eta$ of ${\mathcal O}_X(S)$
does not vanish identically. On the other hand, $\eta$ vanishes at every $y\, \in\, S$, because
$(Q\otimes K_X)_y$ is an eigen-space for ${\rm Res}(D^1, y)$. Therefore, we conclude that
the section $\eta$ of ${\mathcal O}_X(S)$ does not vanish at any point of $X\setminus S$.

Since we have a canonical identification of triples
\begin{equation}\label{idf}
(J^1(Q),\, K_X\otimes {\mathcal O}_X(S),\, D^1)\vert_{X\setminus S}\,=\,
(F',\, L',\, D')\vert_{X\setminus S}\, ,
\end{equation}
where $L'$ is the line subbundle constructed in \eqref{lp}, it can be
deduced that
the second fundamental form of $L'$ for the connection $D'$
$${\mathcal F}(L', D')\, \in\, H^0(X,\, \text{Hom}(L',\, (F'/L')\otimes K_X))
\,=\, H^0(X,\, \text{Hom}(L',\, Q\otimes K_X))$$
does not vanish at any point of $X\setminus S$ (the isomorphism $F'/L' \,=\, Q$ used
above is constructed in \eqref{itp}). Indeed, $\eta\vert_{X\setminus S}\,=\,
{\mathcal F}(L', D')_{X\setminus S}$ (constructed in \eqref{deta})
using the identification in \eqref{idf} (recall that $\eta$ is the second
fundamental form of the subbundle $Q\otimes K_X$ for $D^1$). It
was observed above that $\eta$ does not vanish on $X\setminus S$; therefore,
${\mathcal F}(L', D')$ does not vanish at any point of $X\setminus S$.
However, from the constructions of
$L'$ and $D'$ it follows that ${\mathcal F}(L', D')$ vanishes on $S$.

On the other hand, from \eqref{it2} it follows that
$\text{degree}(\text{Hom}(L',\, Q\otimes K_X))\,=\, d$. Therefore, the divisor
for the section ${\mathcal F}(L', D')$ is exactly $S$ with multiplicity
one. In view of Lemma \ref{lem1},
the triple $(F',\, L',\, D')$ produces a branched projective structure on $X$ with branching
divisor $S$. This completes the proof of the theorem.
\end{proof}

\begin{remark}\label{remnu} While Theorem \ref{thm1} produces a logarithmic connection on 
$J^1(Q)$, satisfying four conditions, when we are given a branched projective structure on 
$X$ with branching divisor $S$, there are many logarithmic connection on $J^1(Q)$, satisfying 
the four conditions, that produce the same branched projective structure on 
$X$. Indeed, if two logarithmic connections on $J^1(Q)$, satisfying
the four conditions, differ by a holomorphic automorphism of $J^1(Q)$, then they
produce the same branched projective structure. On the other hand, given a branched
projective structure on $X$ with branching divisor $S$, the logarithmic connection on
$J^1(Q)$ given by Theorem \ref{thm1} is clearly unique.
\end{remark}

\section{Logarithmic connections on jet bundle and quadratic forms with simple poles}\label{s5}

\subsection{Existence of logarithmic connection on the jet bundle}

Take $Q$ satisfying \eqref{dq}. Consider the jet bundle $J^1(Q)$ in \eqref{e6}.
For every $x_i\, \in\, S$, fix a complex line
$$
\ell_i\, \subset\, J^1(Q)_{x_i}
$$
different from the line $(Q\otimes K_X)_{x_i}\, \subset\, J^1(Q)_{x_i}$ in \eqref{e6};
as before, we identify $\iota_0(Q\otimes K_X)$ with $Q\otimes K_X$ using $\iota_0$. Therefore, we have
\begin{equation}\label{e7}
J^1(Q)_{x_i}\,=\, \ell_i\oplus (Q\otimes K_X)_{x_i}\, .
\end{equation}
Let
\begin{equation}\label{e8}
A_i\, :\, J^1(Q)_{x_i}\,\longrightarrow\, J^1(Q)_{x_i}\, ,\ \ (v_1,\, v_2)\, \longmapsto\,
(0,\, v_2)
\end{equation}
be the projection to the direct summand $(Q\otimes K_X)_{x_i}\, \subset\, J^1(Q)_{x_i}$ in \eqref{e7}.

\begin{proposition}\label{prop1}
There is a logarithmic connection $D$ on $J^1(Q)$, nonsingular over $X\setminus S$,
which satisfies the residue condition ${\rm Res}(D,x_i)\,=\, -A_i$ at each $x_i\, \in \, S$,
where $A_i$ is defined in \eqref{e8}.
\end{proposition}

\begin{proof}
We follow closely the proof Proposition 4.1 in \cite[p.~86]{BDP}. It should be clarified that
Proposition 4.1 in \cite{BDP} is not directly applicable in our
present set-up because the residues $A_i$
do not satisfy the crucial rigidity condition in \cite[Proposition 4.1]{BDP}.

Once we have fixed the residues to be $-A_i$, there is a short exact sequence of
vector bundles on $X$
\begin{equation}\label{e9}
0\, \longrightarrow\, \text{End}(J^1(Q))\otimes{\mathcal O}_X(-S)\, \longrightarrow\, {\mathcal V}\,
\longrightarrow\, TX\otimes {\mathcal O}_X(-S)\, \longrightarrow\, 0
\end{equation}
such that there is a logarithmic connection on $J^1(Q)$ satisfying the conditions in the
statement of the proposition if and only if the short exact sequence in \eqref{e9} splits
holomorphically (see \cite[p.~81, (2.7)]{BDP} and \cite[p.~81, Lemma 2.2]{BDP} for the
construction of $\mathcal V$ and this property of it). So we need
to show that the exact sequence in \eqref{e9} splits holomorphically.

Let
$$
\phi_Q\, \in\, H^1(X,\, \text{Hom}(TX\otimes {\mathcal O}_X(-S),\, \text{End}(J^1(Q))\otimes{\mathcal O}_X(-S)))
\,=\, H^1(X,\, \text{End}(J^1(Q))\otimes K_X)
$$
be the extension class for the short exact sequence in \eqref{e9}. Consider Serre duality
$$
H^1(X,\, \text{End}(J^1(Q))\otimes K_X)\,=\, H^0(X,\, \text{End}(J^1(Q))^*)^*\,=\,
H^0(X,\, \text{End}(J^1(Q)))^*\, ;
$$
note that $\text{End}(J^1(Q))^*\,=\, \text{End}(J^1(Q))$. Let
\begin{equation}\label{e10}
\widetilde{\phi}_Q\, \in\, H^0(X,\, \text{End}(J^1(Q)))^*
\end{equation}
be the image, under this isomorphism, of the above extension class $\phi_Q$ (see \cite[p.~83, (3.7)]{BDP}).

Take any
$$
\tau \, \in\, H^0(X,\, \text{End}(J^1(Q)))\, .
$$
We will show that
\begin{equation}\label{e11}
\tau\, =\, c\cdot\text{Id}_{J^1(Q)}+ N\, ,
\end{equation}
where $c\, \in\, \mathbb C$, and $N$ is a nilpotent endomorphism of $J^1(Q)$ with
$$N(J^1(Q))\, \subset\, Q\otimes K_X\ \ \text{ and } \ \ N(Q\otimes K_X)\,=\,0\, ,$$
where $Q\otimes K_X\, \subset\, J^1(Q)$ is the subbundle in \eqref{e6}.

To prove \eqref{e11}, we will first show that $J^1(Q)$ does not decompose
holomorphically into a direct sum of two holomorphic line bundles. To prove the
indecomposability of $J^1(Q)$ by contradiction, assume that
$$
J^1(Q)\,=\, L_1\oplus L_2\, ,
$$
where $L_1$ and $L_2$ are holomorphic line bundles
with $\text{degree}(L_1) \, \geq\, \text{degree}(L_2)$. If
$$\text{degree}(L_1) \, > \,\text{degree}(Q)\, ,$$ then the composition homomorphism
$$
L_1 \, \hookrightarrow\, J^1(Q) \, \stackrel{q_0}{\longrightarrow}\, Q
$$
is zero, where $q_0$ is the projection in \eqref{e6}. In that case, the
composition
$$
L_2 \, \hookrightarrow\, J^1(Q) \, \stackrel{q_0}{\longrightarrow}\, Q
$$
must be an isomorphism, and consequently the short exact sequence in \eqref{e6} would
split holomorphically.

If, on the other hand, $\text{degree}(L_1) \, \leq \,\text{degree}(Q)$,
consider the inequality
$$
2\cdot \text{degree}(Q)\, \leq\, 2\cdot \text{degree}(Q)+\text{degree}(K_X)\,=\,
\text{degree}(J^1(Q))\,=\, \text{degree}(L_1) + \text{degree}(L_2)
$$
(recall that $g\, \geq\, 1$).
Since $\text{degree}(L_1) \, \geq\, \text{degree}(L_2)$, this implies that
$$
\text{degree}(L_1) \, = \, \text{degree}(Q)\,=\, \text{degree}(L_2)\,=\,
\text{degree}(Q\otimes K_X)\, .
$$
Now take $i\, \in\, \{1,\, 2\}$ such that the composition
$$
L_i \, \hookrightarrow\, J^1(Q) \, \stackrel{q_0}{\longrightarrow}\, Q
$$
is nonzero. Since $\text{degree}(L_i) \, = \, \text{degree}(Q)$, this composition
is an isomorphism, because it is nonzero. So if $\text{degree}(L_1) \, \leq \,\text{degree}(Q)$ we
again conclude that the short exact sequence in \eqref{e6}
splits holomorphically.

Hence the short exact sequence in \eqref{e6} splits holomorphically. But this means that $Q$ 
admits a holomorphic connection, because the homomorphism $J^1(Q)\, \longrightarrow\, Q\otimes 
K_X$ given by such a splitting defines a holomorphic differential operator of order one (see 
Section \ref{sdo}) that is indeed a holomorphic connection on $Q$. This in turn implies that 
$\text{degree}(Q)\,=\, 0$ (Remark \ref{flat}). But this contradicts \eqref{dq2}.

Therefore, we conclude that the vector bundle $J^1(Q)$ does not decompose into a direct
sum of two holomorphic line bundles. This implies that $\tau$ in \eqref{e11} is of the
form
$$
\tau \, =\, c\cdot\text{Id}_{J^1(Q)}+ N\, ,
$$
where $c\, \in\, \mathbb C$ and $N$ is nilpotent \cite[p.~201, Proposition 15]{At}.
Since $N$ is nilpotent, to prove \eqref{e11} it suffices to show that
$$N(Q\otimes K_X)\,\subset\, \,Q\otimes K_X\, ,$$
where $Q\otimes K_X\, \subset\, J^1(Q)$ is the subbundle in \eqref{e6}. If
$N(Q\otimes K_X)\,\subsetneq\, Q\otimes K_X$, then the composition of sheaf
homomorphisms
$$
Q\otimes K_X\,\stackrel{N}{\hookrightarrow} \, J^1(Q) \stackrel{q_0}{\longrightarrow}\, Q
$$
is nonzero, where $q_0$ is the projection in \eqref{e6}.
But this implies that this composition is an isomorphism because
$\text{degree}(Q\otimes K_X)\, \geq\, \text{degree}(Q)$ (recall that $g\, \geq\, 1$). Therefore, 
the short exact sequence in \eqref{e6} splits holomorphically, which 
contradicts \eqref{dq2} as before.

This completes the proof of \eqref{e11}.

Consider the functional $\widetilde{\phi}_Q$ in \eqref{e10}. We have
$$
\widetilde{\phi}_Q(\text{Id}_{J^1(Q)}) \,=\,
\text{degree}(J^1(Q))+\sum_{i=1}^d \text{trace}(-A_i)
$$
(see the proof of Lemma 3.2 in \cite{BDP}). Consequently, from \eqref{det} and \eqref{e8} it
follows immediately that
$$
\widetilde{\phi}_Q(\text{Id}_{J^1(Q)})\,=\, d-d\,=\, 0\, .
$$

Now take any $N$ as in \eqref{e11}. Since
$N(J^1(Q))\, \subset\, Q\otimes K_X$, from the definition of $A_i$ it follows immediately that
$$\text{trace}(A_i\circ N(x_i))\,=\, 0$$ for all $i$; note that in fact 
$N(x_i)\circ A_i\,=\, 0$ for all $i$.
Now as in the proof Proposition 4.1 in \cite[p.~86]{BDP} we conclude that
$\widetilde{\phi}_Q(N)\,=\, 0$. Consequently, from \eqref{e11} it follows that
$\widetilde{\phi}_Q\,=\, 0$. This implies that ${\phi}_Q\,=\, 0$.
Consequently, the short exact sequence in \eqref{e9} splits holomorphically.
Hence there is a logarithmic connection $D$ on $J^1(Q)$, which is nonsingular over $X\setminus S$,
such that ${\rm Res}(D,x_i)\,=\, -A_i$ at each $x_i\, \in \, S$.
\end{proof}

The following is a rather straight-forward consequence of Proposition \ref{prop1}.

\begin{corollary}\label{cor1}
Fix $A_i$ as in \eqref{e8}.
There is a logarithmic connection $D$ on $J^1(Q)$, nonsingular over $X\setminus S$, such that
\begin{enumerate}
\item ${\rm Res}(D,x_i)\,=\, -A_i$ at each $x_i\, \in \, S$, and

\item the logarithmic connection on $\bigwedge\nolimits^2 J^1(Q)\,=\, {\mathcal O}_X(S)$ induced by
$D$ coincides with the canonical logarithmic connection ${\mathcal D}_S$ given by the de
Rham differential.
\end{enumerate}
\end{corollary}

\begin{proof}
Let $D_1$ be a logarithmic connection on $J^1(Q)$ given by Proposition \ref{prop1}. The
logarithmic connection on $\bigwedge\nolimits^2 J^1(Q)\,=\, {\mathcal O}_X(S)$ induced by
$D_1$ will be denoted by $D'_1$. The difference
$$
D'_1 - {\mathcal D}_S\, \in\, H^0(X,\, K_X\otimes{\mathcal O}_X(S))\, .
$$
For any point $y\,\in\, S$, we have
$$
{\rm Res}(D'_1,y)\,=\, \text{trace}({\rm Res}(D_1,y)) \,=\, -\text{trace}(A_i) \,=\, -1
\,=\, {\rm Res}({\mathcal D}_S,y) \, .
$$
Therefore, we have
$$\theta\, :=\, D'_1 - {\mathcal D}_S\, \in \, H^0(X,\, K_X)\, \subset\,
H^0(X,\, K_X\otimes{\mathcal O}_X(S))\, .$$
Now it is straight-forward to check that the logarithmic connection
$$
D\, =\, D_1 - \frac{\theta}{2}\cdot \text{Id}_{J^1(Q)}
$$
on $J^1(Q)$ satisfies all the conditions in the statement of the corollary.
\end{proof}

\subsection{Quadratic forms with simple poles at $S$}\label{quadratic forms}

Let ${\mathcal C}(Q)$ denote the space of all logarithmic connections $D^1$ on $J^1(Q)$ satisfying the
following conditions:
\begin{enumerate}
\item $D^1$ is nonsingular over $X\setminus S$,

\item ${\rm trace}({\rm Res}(D^1, x_i))\, =\, -1$ and ${\rm Res}(D^1,x_i)(w)\,=\, -w$
for every $x_i\, \in \, S$ and $w\, \in\, (Q\otimes K_X)_{x_i}$, where
$(Q\otimes K_X)_{x_i}\, \subset\, J^1(Q)_{x_i}$ is the line in \eqref{e6}, and

\item the logarithmic connection on $\bigwedge\nolimits^2 J^1(Q)\,=\, {\mathcal O}_X(S)$ induced by
$D^1$ coincides with the canonical logarithmic connection ${\mathcal D}_S$.
\end{enumerate}

From Theorem \ref{thm1} we know that an element $D^1 \, \in\, {\mathcal C}(Q)$ corresponds to a
branched projective structure on $X$ with branching type $S$ if and only if
the homomorphism $\rho(D^1, y)$ in Proposition \ref{propn1} vanishes for every $y\,\in\, S$.
As noted in Remark \ref{remnu}, this correspondence is not bijective.

The space of all logarithmic connections on $J^1(Q)$ singular over $S$ is an affine space for the vector space 
$H^0(X,\, \text{End}(J^1(Q))\otimes K_X\otimes {\mathcal O}_X(S))$. However, for any $D^1\, \in\, {\mathcal 
C}(Q)$ and $\omega\, \in\, H^0(X,\, \text{End}(J^1(Q))\otimes K_X\otimes {\mathcal O}_X(S))$, the
logarithmic connection $D^1+\omega$ may not satisfy the conditions (2) and (3) in the above definition of
${\mathcal C}(Q)$. We note that $D^1+\omega$ satisfies condition (3) in the definition of
${\mathcal C}(Q)$ if and only if $\text{trace}(\omega)\, \in\, H^0(X,\, K_X\otimes {\mathcal O}_X(S))$
vanishes.

Let $\text{Aut}(J^1(Q))$ denote the group of all holomorphic automorphisms $$T\, :\, J^1(Q)\, \longrightarrow\, 
J^1(Q)$$ such that the induced automorphism $\bigwedge\nolimits^2 T\, :\, \bigwedge\nolimits^2 J^1(Q)\, 
\longrightarrow\, \bigwedge\nolimits^2 J^1(Q)$ is the identity map of the line bundle $\bigwedge\nolimits^2 J^1(Q)$. 
This group $\text{Aut}(J^1(Q))$ has a natural action on the space of all logarithmic connections on $J^1(Q)$ singular 
over $S$. The action of any $T\, \in\, {\rm Aut}(J^1(Q))$ on any logarithmic connection $D'$ on $J^1(Q)$, singular over 
$S$, will be denoted by $T\circ D'$.

Since $T$ induces the trivial automorphism of $\bigwedge^2 J^1(Q)$, the two logarithmic connections
$D$ and $T\circ D'$ induce the same logarithmic connection on $\bigwedge^2 J^1(Q)$.

\begin{lemma}\label{lem2}
The natural action of ${\rm Aut}(J^1(Q))$ on the
space of all logarithmic connections on $J^1(Q)$ singular over $S$ preserves the subset
${\mathcal C}(Q)$ defined above.
\end{lemma}

\begin{proof}
In the proof of Proposition \ref{prop1} we saw that any $T\, \in\, {\rm Aut}(J^1(Q))$ is
of the form $c\cdot\text{Id}_{J^1(Q)}+ N$,
where $c = \pm 1$ (since $T$ acts trivially on $\bigwedge^2 J^1(Q)$) and $N$ is a nilpotent endomorphism of $J^1(Q)$ with
$N(J^1(Q))\, \subset\, Q\otimes K_X$ (see \eqref{e11}). Using this it is straightforward to
check for any $D'\, \in\, {\mathcal C}(Q)$, the logarithmic connection $T\circ D'$ also satisfies
the residue conditions in the definition of ${\mathcal C}(Q)$.
\end{proof}

From Lemma \ref{lem2} we know that the group ${\rm Aut}(J^1(Q))$ acts on ${\mathcal C}(Q)$.

\begin{lemma}\label{lem3}
Take any $D'\, \in\, {\mathcal C}(Q)$ and $T\, \in\, {\rm Aut}(J^1(Q))$ such that
$T\circ D'\,=\, D'$. Then $T\, =\, \pm {\rm Id}_{J^1(Q)}$.
\end{lemma}

\begin{proof}
Let $\widetilde{D}'$ be the logarithmic connection on $\text{End}(J^1(Q))$ induced by the
logarithmic connection $D'$ on $J^1(Q)$.

Take $T\,=\, c\cdot\text{Id}_{J^1(Q)}+ N$, where $c\,=\,\pm 1$ and $N$ is as in
\eqref{e11}. Since $T\circ D'\,=\, D'$, it follows that the section $N$ of $\text{End}(J^1(Q))$
is flat with respect to the logarithmic connection $\widetilde{D}'$ on $\text{End}(J^1(Q))$. Therefore,
$\text{image}(N)\, \subset\, J^1(Q)$ is preserved by the logarithmic connection $D'$. If $N\, \not=\, 0$, then
$\text{image}(N)$ generates the subbundle $\iota_0(Q\otimes K_X)\, \subset\, J^1(Q)$ in \eqref{e6}. In the
proof of Theorem \ref{thm1} we saw that $D'$ does not preserve the subbundle $\iota_0(Q\otimes K_X)$
(see \textbf{Claim A}). Hence we have $N\,=\, 0$.
\end{proof}

Since $\pm {\rm Id}_{J^1(Q)}\, \subset\, {\rm Aut}(J^1(Q))$ acts trivially on the logarithmic connections on
$J^1(Q)$, the above action of ${\rm Aut}(J^1(Q))$ on ${\mathcal C}(Q)$ produces an action of
$$
{\rm Aut}'(J^1(Q))\,:=\, {\rm Aut}(J^1(Q))/(\pm {\rm Id}_{J^1(Q)})
$$
on ${\mathcal C}(Q)$.

The following is a direct consequence of Lemma \ref{lem3}:

\begin{corollary}\label{cor2}
The action of ${\rm Aut}'(J^1(Q))$ on ${\mathcal C}(Q)$ is free.
\end{corollary}

The orbit ${\rm Aut}'(J^1(Q))\circ D'\, \subset\, {\mathcal C}(Q)$ of any $D'\, \in\, {\mathcal C}(Q)$
under the action of ${\rm Aut}'(J^1(Q))$ will be denoted by ${\rm Orb}(D')$.

Consider the injective homomorphism of vector bundles
$$
\psi' \,:\, K_X\,=\, \text{Hom}(Q,\, Q\otimes K_X)\, \longrightarrow\, \text{End}(J^1(Q))\, , $$
defined by $$\ \ t \longmapsto (v\,
\longmapsto \, \iota_0(t(q_0(v))))\, , \ \ \forall \ t \, \in\, \text{Hom}(Q,\, Q\otimes K_X)_x
\, , \ \ \forall \ v\, \in\, J^1(Q)_x\, , \ \ \forall \ x\,\in\, X\, ,$$
where $\iota_0$ and $q_0$ are the homomorphisms in \eqref{e6}. This produces a homomorphism
$$
\psi\,=\, \psi'\otimes 
{\rm Id} \, :\, K^{\otimes 2}_X\otimes {\mathcal O}_X(S)\,=\, K_X\otimes K_X\otimes {\mathcal O}_X(S)\, 
\longrightarrow\,\text{End}(J^1(Q))\otimes K_X\otimes {\mathcal O}_X(S)\, ,
$$
where ${\rm Id}$ stands for the identity map of $K_X\otimes {\mathcal O}_X(S)$. Let
\begin{equation}\label{e13}
\psi_*\, :\, H^0(X,\, K^{\otimes 2}_X\otimes {\mathcal O}_X(S))\, \longrightarrow\,
H^0(X,\, \text{End}(J^1(Q))\otimes K_X\otimes {\mathcal O}_X(S))
\end{equation}
be the homomorphism of global sections induced by this $\psi$.

\begin{proposition}\label{prop2}
Take any $D\, \in\, {\mathcal C}(Q)$ and $\theta\, \in\, H^0(X,\, K^{\otimes 2}_X\otimes {\mathcal O}_X(S))$. 
\begin{enumerate}
\item Then
$$
D+\psi_*(\theta)\, \in\, {\mathcal C}(Q)\, ,
$$
where $\psi_*$ is constructed in \eqref{e13}.

\item The map
$$
\Gamma_D\, :\, H^0(X,\, K^{\otimes 2}_X\otimes {\mathcal O}_X(S))\, \longrightarrow\,
{\mathcal C}(Q)\, ,\ \ \theta\, \longmapsto\, D+\psi_*(\theta)
$$
is an embedding.

\item For any $D'\, \in\, {\rm image}(\Gamma_D)$,
$$
{\rm image}(\Gamma_D)\cap {\rm Orb}(D')\,=\, \{D'\}\, .
$$
\end{enumerate}
\end{proposition}

\begin{proof}
For any $y\, \in\, S$, we have, by definition of $\psi'$ and $\psi$, that
$$
\psi_*(\theta)(y)\,=\, N_y\otimes u\, ,
$$
where $u\, \in\, (K_X \otimes {\mathcal O}_X(S))_y$ and $N_y\, \in\, \text{End}(J^1(Q))_y$ is a nilpotent endomorphism
with $N_y(\iota_0(Q\otimes K_X)_y)\,=\, 0$
(the homomorphism $\iota_0$ is the one in \eqref{e6}).
Since $N_y(\iota_0(Q\otimes K_X)_y)\,=\, 0$ and $\text{trace}(\psi_*(\theta))\,=\, 0$,
 it follows immediately that $D+\psi_*(\theta)$ satisfies the residue conditions in the
definition of ${\mathcal C}(Q)$. Also the logarithmic connections on $\bigwedge\nolimits^2 J^1(Q)$ induced by $D$ and
$D+\psi_*(\theta)$ coincide, because $\text{trace}(\psi_*(\theta))\,=\, 0$.
Hence we have $D+\psi_*(\theta)\, \in\, {\mathcal C}(Q)$.

The map $\Gamma_D$ is evidently an embedding as $\psi_*$ is injective.

To prove the final part of the proposition, note that for any $T\, \in\, {\rm Aut}(J^1(Q))$
and $\theta\, \in\, H^0(X,\, K^{\otimes 2}_X\otimes {\mathcal O}_X(S))$, we have
\begin{equation}\label{c1}
\psi_*(\theta)\circ T\,=\, (T\otimes \text{Id}_{K_X \otimes {\mathcal O}_X(S)})\circ \psi_*(\theta)\, ;
\end{equation}
both sides are homomorphisms from $J^1(Q)$ to $J^1(Q)\otimes K_X\otimes {\mathcal O}_X(S)$.
The equality in \eqref{c1} follows immediately form the fact that $T$ is of the form
$c\cdot\text{Id}_{J^1(Q)}+ N$,
where $c \,=\, \pm 1$ and $N$ is a nilpotent endomorphism of $J^1(Q)$ with
$N(J^1(Q))\, \subset\, Q\otimes K_X$.

Take $D'\,=\, D+ \psi_*(\theta_1)\, \in\, {\rm image}(\Gamma_D)$ and
any $T_1\, \in\, {\rm Aut}(J^1(Q))$ such that
\begin{equation}\label{c2}
T\circ D'\,=\, D+\psi_*(\theta_2)\, \in\, {\rm image}(\Gamma_D)\, ,
\end{equation}
where $\theta_1,\, \theta_2\, \in\, H^0(X,\, K^{\otimes 2}_X\otimes {\mathcal O}_X(S))$. From
\eqref{c1} we have $T\circ D'\,=\, T\circ D + \psi_*(\theta_1)$, so from \eqref{c2}
it follows that
\begin{equation}\label{e14}
T\circ D \,=\, D+\psi_*(\theta_2)- \psi_*(\theta_1)\,=\, D+\psi_*(\theta_2-\theta_1)\, .
\end{equation}

Write $T$ in \eqref{e14} as
\begin{equation}\label{e15}
T\,=\, c\cdot\text{Id}_{J^1(Q)}+ N\, ,
\end{equation}
where $c\,=\,\pm 1$ and $N$ is nilpotent with
$N(J^1(Q))\, \subset\, \iota_0(Q\otimes K_X)$; the homomorphism $\iota_0$ is
the one in \eqref{e6}. This implies that
$T^{-1}\,=\, c^{-1}\cdot\text{Id}_{J^1(Q)}- N$. Hence we have
$$
(T\circ D)(s) \,=\,D(s) + c(D(N(s))- N(D(s))) - N(D(N(s)))
$$
for every locally defined holomorphic section $s$ of $J^1(Q)$. Therefore, from \eqref{e14}, we have
\begin{equation}\label{e16}
c(D(N(s))- N(D(s))) - N(D(N(s)))\,=\, (\psi_*(\theta_2-\theta_1))(s)\, .
\end{equation}
Now take $s$ to be a locally defined holomorphic section of $J^1(Q)$ that does not lie in the image
of the homomorphism $\iota_0$ in \eqref{e6}. Note that
$-cN(D(s)) - N(D(N(s)))$ lies in $\iota_0(Q\otimes K_X)\otimes K_X\otimes {\mathcal O}_X(S)$. Also,
$(\psi_*(\theta_2-\theta_1))(s)$ lies in $\iota_0(Q\otimes K_X)\otimes K_X\otimes {\mathcal O}_X(S)$, because
$$\text{image}(\psi)(J^1(Q))\, \subset\,J^1(Q)\otimes K_X\otimes {\mathcal O}_X(S)$$
is actually contained in $\iota_0(Q\otimes K_X)\otimes K_X\otimes {\mathcal O}_X(S)$. Hence from
\eqref{e16} it follows that
$D(N(s))$ lies in $\iota_0(Q\otimes K_X)\otimes K_X\otimes {\mathcal O}_X(S)$.

Assume that $N$ is nonzero. Then the subbundle of $J^1(Q)$ generated by the image of $N$ is
$\iota_0(Q\otimes K_X)$. Since $D(N(s))$ lies in $\iota_0(Q\otimes K_X)\otimes K_X\otimes {\mathcal O}_X(S)$,
it now follows that $D$ preserves the subbundle $\iota_0(Q\otimes K_X)\, \subset\, J^1(Q)$.
But we saw in \textbf{Claim A} in the proof of Theorem \ref{thm1} that
$D$ does not preserve $\iota_0(Q\otimes K_X)$.

Hence we conclude that $N\,=\, 0$. This implies that $T\circ D\,=\, D$. Now from \eqref{e14} it follows that 
$\theta_1\,=\, \theta_2$.
\end{proof}

\section{A special class of logarithmic connections on jet bundle}\label{s6}

Consider the vector bundle $J^1(Q)$, where $Q$ is the line bundle in \eqref{dq}. 
Let
\begin{equation}\label{e17}
0\, \longrightarrow\, J^1(Q)\otimes K_X \, \stackrel{\iota_1}{\longrightarrow}\, J^1(J^1(Q))\,
\stackrel{q_1}{\longrightarrow}\, J^0(J^1(Q)) \,=\, J^1(Q) \, \longrightarrow\, 0
\end{equation}
be the jet sequence for $J^1(Q)$ as in \eqref{je1}. Any holomorphic homomorphism of vector bundles
$$h\, :\, V\, \longrightarrow\, W$$ produces a homomorphism
$$
h_*\, :\,J^1(V)\, \longrightarrow\, J^1(W)
$$
of jet bundles. Let
\begin{equation}\label{e18}
q_{0*}\, :\, J^1(J^1(Q))\, \longrightarrow\, J^1(Q)
\end{equation}
be the homomorphism produced by the projection $q_0$ in \eqref{e6}. It is straight-forward to
check that $$q_0\circ q_{0*}\,=\, q_0\circ q_1\, \in\, H^0(X,\, \text{Hom}(J^1(J^1(Q)),\, Q))\, ,$$
where $q_1$ is the homomorphism in \eqref{e17}. Therefore, we have
\begin{equation}\label{e19}
q_1- q_{0*}\, :\, J^1(J^1(Q))\, \longrightarrow\, \text{kernel}(q_0)\,=\, Q\otimes K_X\, .
\end{equation}

Take a logarithmic connection $D$ on $J^1(Q)$ singular over $S$. So
$D$ produces a homomorphism
\begin{equation}\label{td}
\widetilde{D}\, :\, J^1(J^1(Q))\, \longrightarrow\,J^1(Q)\otimes K_X\otimes {\mathcal O}_X(S)
\end{equation}
such that $\widetilde{D}\circ \iota_1$ is the homomorphism $J^1(Q)\otimes K_X\, \longrightarrow\,
J^1(Q)\otimes K_X\otimes {\mathcal O}_X(S)$ given by the natural inclusion of sheaves, where $\iota_1$
is the homomorphism in \eqref{e17}. Define
\begin{equation}\label{kd}
K(\widetilde{D})\,:=\,\text{kernel}(\widetilde{D})\, \subset\, J^1(J^1(Q))\, .
\end{equation}
The restriction
$$
q_1\vert_{K(\widetilde{D})}\, :\, K(\widetilde{D})\, \longrightarrow\, J^1(Q)
$$
is an isomorphism over the complement
\begin{equation}\label{e20}
X'\, :=\, X\setminus S\, ,
\end{equation}
where $q_1$ is the homomorphism in \eqref{e17}; this follows form the fact $\widetilde{D}\circ\iota_1$
is the natural inclusion of $J^1(Q)\otimes K_X$ in $J^1(Q)\otimes K_X\otimes {\mathcal O}_X(S)$.
Hence the restriction $\widetilde{D}\vert_{X'}$ is a regular connection on $J^1(Q)\vert_{X'}$. Let
\begin{equation}\label{e21}
q^D_1\, :=\, (q_1\vert_{K(\widetilde{D})})\vert_{X'}\, :\, K(\widetilde{D})\vert_{X'}\,
\longrightarrow\, J^1(Q)\vert_{X'}
\end{equation}
be this isomorphism over $X'$, where $K(\widetilde{D})$ is constructed in \eqref{kd}.

\begin{lemma}\label{lemext}
The homomorphism
$$
((q_1- q_{0*})\vert_{X'})\circ (q^D_1)^{-1}\, :\, J^1(Q)\vert_{X'}\, \longrightarrow\,
(Q\otimes K_X)\vert_{X'}
$$
over $X'\, \subset\, X$, where $q_1- q_{0*}$ and $q^D_1$ are the homomorphisms constructed in \eqref{e19} and \eqref{e21}
respectively, extends to a homomorphism
$$
J^1(Q)\, \longrightarrow\,Q\otimes K_X\otimes {\mathcal O}_X(S)
$$
over entire $X$.
\end{lemma}

\begin{proof}
Consider the natural inclusion of sheaves $J^1(Q)\otimes K_X\, \hookrightarrow\, J^1(Q)\otimes K_X\otimes
{\mathcal O}_X(S)$. Using it, we have the injective homomorphism of vector bundles
$$
J^1(Q)\otimes K_X\, \longrightarrow\, (J^1(Q)\otimes K_X\otimes {\mathcal O}_X(S)) \oplus
J^1(J^1(Q)),\, \ \ v\, \longmapsto\, (v,\, -\iota_1(v))\, ,
$$
where $\iota_1$ is the homomorphism in \eqref{e17}. The corresponding quotient
$$
{\mathcal H}\, :=\, ((J^1(Q)\otimes K_X\otimes {\mathcal O}_X(S)) \oplus J^1(J^1(Q)))/(J^1(Q)\otimes K_X)
$$
is a vector bundle, because $\iota_1$ is fiberwise injective. The holomorphic vector bundle $\mathcal H$
fits in the short exact sequence
\begin{equation}\label{shl}
0\, \longrightarrow\, J^1(Q)\otimes K_X\otimes {\mathcal O}_X(S)\, \stackrel{\iota'_1}{\longrightarrow}\, {\mathcal H}
\,\stackrel{q'_1}{\longrightarrow}\, J^1(Q) \, \longrightarrow\, 0\, ,
\end{equation}
where $q'_1$ is induced by the projection
$$
(J^1(Q)\otimes K_X\otimes {\mathcal O}_X(S)) \oplus J^1(J^1(Q))\, \longrightarrow\, J^1(Q),\, \ \
v\, \longmapsto\, (0, q_1(v))
$$
($q_1$ is the homomorphism in \eqref{e17}),
while $\iota'_1$ sends any $v\, \in\, J^1(Q)\otimes K_X\otimes {\mathcal O}_X(S)$ to the image of
$(v,\, 0)\, \in\, (J^1(Q)\otimes K_X\otimes {\mathcal O}_X(S)) \oplus J^1(J^1(Q))$ in the quotient
bundle ${\mathcal H}$. Also, note that
$$
{\mathcal H}\vert_{X'}\,=\, J^1(J^1(Q))\vert_{X'}\, .
$$

Consider the homomorphism
$$
(J^1(Q)\otimes K_X\otimes {\mathcal O}_X(S)) \oplus J^1(J^1(Q))\, \longrightarrow\, J^1(Q)\otimes K_X\otimes {\mathcal
O}_X(S)\, , \ \ (v,\, w) \, \longmapsto\, v+\widetilde{D}(w)\, ,
$$
where $\widetilde{D}$ is the homomorphism in \eqref{td}. In descends to a homomorphism
$$
\widetilde{D}_1\, :\, {\mathcal H}\, \longrightarrow\,J^1(Q)\otimes K_X\otimes {\mathcal O}_X(S)
$$
which satisfies the equation $\widetilde{D}_1\circ\iota'_1\,=\,
\text{Id}_{J^1(Q)\otimes K_X\otimes {\mathcal O}_X(S)}$, where $\iota'_1$ is the
homomorphism in \eqref{shl}. In other words, the homomorphism $\widetilde{D}_1$ produces
a holomorphic splitting of the short exact sequence in \eqref{shl}. Let
\begin{equation}\label{beta1}
\beta_1\, :\, J^1(Q) \, \longrightarrow\, {\mathcal H}
\end{equation}
be the homomorphism corresponding to this holomorphic
splitting, so we have $$q'_1\circ \beta_1\,=\, \text{Id}_{J^1(Q)}\, ,$$ where $q'_1$
is the homomorphism in \eqref{shl}.

Using the homomorphism $\iota_1$ in \eqref{e17} and the natural inclusion of sheaves
$$
J^1(J^1(Q))\, \hookrightarrow\, J^1(J^1(Q))\otimes {\mathcal O}_X(S)\, ,
$$
we have the homomorphism
$$
(J^1(Q)\otimes K_X\otimes {\mathcal O}_X(S)) \oplus J^1(J^1(Q))\, \longrightarrow\,
J^1(J^1(Q))\otimes {\mathcal O}_X(S)\, , \ \ (v,\, w)\, \longmapsto\, (\iota_1\otimes {\rm Id})(v)+w\, ,
$$
where ${\rm Id}$ stands for the identity map of ${\mathcal O}_X(S)$. This descends to a homomorphism
$$
\beta_2\, :\, {\mathcal H}\, \longrightarrow\, J^1(J^1(Q))\otimes {\mathcal O}_X(S)\, .
$$
Now we have the composition homomorphism
$$
\beta_2 \circ\beta_1\, :\, J^1(Q) \, \longrightarrow\, J^1(J^1(Q))\otimes {\mathcal O}_X(S)\, ,
$$
where $\beta_1$ is constructed in \eqref{beta1}. Note that $$(\beta_2 \circ\beta_1)\vert_{X'}\,=\,
(q^D_1)^{-1} \, ,$$
where $q^D_1$ is constructed in \eqref{e21}. Finally, consider
$$
((q_1- q_{0*})\otimes {\rm Id}) \circ (\beta_2 \circ\beta_1)\, :\, J^1(Q) \, \longrightarrow\,
Q\otimes K_X\otimes {\mathcal O}_X(S)\, ,
$$
where $q_1- q_{0*}$ is constructed in \eqref{e19}, and ${\rm Id}$ stands for the identity map of ${\mathcal O}_X(S)$.
The restriction of this homomorphism over $X'$ clearly coincides with $((q_1- q_{0*})\vert_{X'})\circ (q^D_1)^{-1}$.
\end{proof}

\begin{definition}\label{def1}
A \textit{special} logarithmic connection on $J^1(Q)$ is a logarithmic connection
$D\, \in\, {\mathcal C}(Q)$ (see Section \ref{quadratic forms} for ${\mathcal C}(Q)$) such that
$$
((q_1- q_{0*})\vert_{X'})\circ (q^D_1)^{-1}\, =\, 0\, ,
$$
where $q_1- q_{0*}$ and $q^D_1$ are the homomorphisms constructed in \eqref{e19} and \eqref{e21} 
respectively; equivalently, the homomorphism $J^1(Q) \, \longrightarrow\,
Q\otimes K_X\otimes {\mathcal O}_X(S)$ given by Lemma \ref{lemext}, whose restriction to $X'$
is $((q_1- q_{0*})\vert_{X'})\circ (q^D_1)^{-1}$, vanishes identically.

The space of special logarithmic connections on $J^1(Q)$ will be denoted by ${\mathcal C}^0(Q)$.
\end{definition}

\begin{proposition}\label{prop3}\mbox{}
\begin{enumerate}
\item The set ${\mathcal C}^0(Q)$ in Definition \ref{def1} is non-empty. More precisely,
for any $D\, \in\, {\mathcal C}(Q)$, there is an
element of ${\rm Aut}(J^1(Q))$ that takes $D$ into ${\mathcal C}^0(Q)$ by the action
in Lemma \ref{lem2}.

\item For any given $D\, \in\, {\mathcal C}^0(Q)$, the corresponding subset 
$$
{\rm image}(\Gamma_D)\,=\,
\Gamma_D(H^0(X,\, K^{\otimes 2}_X\otimes {\mathcal O}_X(S)))\, \subset\, {\mathcal C}(Q)
$$
in Proposition \ref{prop2} coincides with ${\mathcal C}^0(Q)$.
\end{enumerate}
\end{proposition}

\begin{proof}
Take a logarithmic connection $D\, \in\, {\mathcal C}(Q)$. We will produce an
element of ${\rm Aut}(J^1(Q))$ that takes $D$ into ${\mathcal C}^0(Q)$ by the action
in Lemma \ref{lem2}.

For any $y\, \in\, S$ consider the residue $\text{Res}(D, y)\,\in\, \text{End}(J^1(Q)_y)$. Let
$$\ell'_y\,:= \, \text{kernel}(\text{Res}(D, y))\, \subset\, J^1(Q)_y$$
be the line. Let $F'$ be the holomorphic vector bundle
over $X$ defined by the exact sequence in \eqref{g2}.
From the second part in the proof of Theorem \ref{thm1} we know that
$\bigwedge\nolimits^2 F'\,=\, {\mathcal O}_X$ (see \eqref{dz}), and the line subbundle
$\iota_0(Q\otimes K_X)\, \subset\, J^1(Q)$ in \eqref{e6} produces a line subbundle
\begin{equation}\label{lp2}
L'\, \subset\, F'
\end{equation}
(see \eqref{lp}), so we have an isomorphism
$$
\iota'\, :\, F'/L'\, \longrightarrow\, Q
$$
as in \eqref{itp}. Let
\begin{equation}\label{qp2}
q'\, :\, F'\, \longrightarrow\, F'/L'\,=\, Q
\end{equation}
be the quotient map.

The logarithmic connection $D$ on $J^1(Q)$ produces a holomorphic connection on $F'$;
this holomorphic connection on $F'$ will be denoted by $D'$, as done in the proof of Theorem \ref{thm1}.

Now we have a homomorphism
$$
\varphi\, :\, F'\, \longrightarrow\, J^1(Q)
$$
as in \eqref{vp} that sends any $v\, \in\, F'_x$ to the element of $J^1(Q)_x$
given by the locally defined section $q'(\widehat{v})$ of $Q$, where $\widehat{v}$ is the unique
flat section of $F'$ for the
connection $D'$, defined around $x\, \in\, X$, such that $\widehat{v}(x)\,=\, v$, and
$q'$ is the projection in \eqref{qp2}.

We claim that there is a unique isomorphism
$$
\xi\, :\, J^1(Q) \, \longrightarrow\, J^1(Q)
$$
such that the following diagram of homomorphisms is commutative:
\begin{equation}\label{cd2}
\begin{matrix}
F' & \stackrel{\varphi}{\longrightarrow} & J^1(Q)\\
\Vert && ~\Big\downarrow\xi \\
F' & \stackrel{\iota}{\longrightarrow} & J^1(Q)
\end{matrix}
\end{equation}
where $\iota$ is the homomorphism in \eqref{g2}.

To prove the above claim, first note that $\varphi$ and $\iota$ are isomorphisms
over $X'$ (defined in \eqref{e20}). Therefore, there is an unique automorphism $\xi'$ of
$J^1(Q)\vert_{X'}$ such that the diagram
$$
\begin{matrix}
F'\vert_{X'} & \stackrel{\varphi}{\longrightarrow} & J^1(Q)\vert_{X'}\\
\Vert && ~\Big\downarrow\xi' \\
F'\vert_{X'} & \stackrel{\iota}{\longrightarrow} & J^1(Q)\vert_{X'}
\end{matrix}
 $$
is commutative.

Next, for any $y\, \in\, S$, the kernel of $\varphi(y)$
coincides with the kernel of $\iota(y)$. Indeed, both the kernels coincide with the
line $L'_y\, \subset\, F'_y$ in \eqref{lp2}. From this it follows immediately that
the above automorphism $\xi'$ of $J^1(Q)\vert_{X'}$ extends to an automorphism of
$J^1(Q)$ such that \eqref{cd2} is commutative. Indeed, the data
$(F',\, \{L'_y\, \subset\, F'_y\}_{y\in S})$ determine the bigger subsheaf $J^1(Q)$ uniquely
using the following exact sequence
$$
0\, \longrightarrow\, J^1(Q)^*\, \longrightarrow\, (F')^*\, \longrightarrow\,
\bigoplus_{y\in S} (L'_y)^* \, \longrightarrow\, 0\, .
$$
We may set the above injective homomorphism of sheaves $J^1(Q)^*\, \longrightarrow\, (F')^*$
to be both $\varphi^*$ and $\iota^*$. After doing that we see that the identity map
of $(F')^*$ produces an automorphism of $J^1(Q)^*$. In other words, there is a commutative
diagram of homomorphisms
$$
\begin{matrix}
0 & \longrightarrow & J^1(Q)^* & \stackrel{\iota^*}{\longrightarrow} & (F')^* & \longrightarrow &
\bigoplus_{y\in S} (L'_y)^* & \longrightarrow & 0\\
&&~ \Big\downarrow\xi_1 && \Vert && \Vert\\
0 & \longrightarrow & J^1(Q)^* & \stackrel{\varphi^*}{\longrightarrow} & (F')^* & \longrightarrow &
\bigoplus_{y\in S} (L'_y)^* & \longrightarrow & 0
\end{matrix}
$$
where $\xi_1$ is an isomorphism. Now set $\xi\, =\, \xi^*_1$; clearly, $\xi\vert_{X'}\,=\, \xi'$.
Therefore, we have shown
that there is a unique automorphism $\xi$ of $J^1(Q)$ such that the diagram in 
\eqref{cd2} is commutative.

As in the first part of the proof of Theorem \ref{thm1}, let $D^1$ denote the logarithmic
connection on $J^1(Q)$ induced by the holomorphic connection $D'$ on $F'$ using the
homomorphism $\varphi$ in \eqref{cd2}. From the commutativity of the diagram in \eqref{cd2} it follows
immediately that the automorphism $\xi$ of $J^1(Q)$ takes the logarithmic connection $D^1$
to the logarithmic connection $D$. Indeed, $\varphi$ and $\iota$ take the connection
$D'$ on $F'$ to $D^1$ and $D$ respectively.

We know from the proof of Theorem \ref{thm1} that $D^1\, \in\, {\mathcal C}(Q)$.
We will now show that
\begin{equation}\label{sd1}
D^1\, \in\, {\mathcal C}^0(Q)\, .
\end{equation}

Over the open subset $X'$ in \eqref{e20}, the homomorphism $\varphi$ in
\eqref{cd2} is an isomorphism. Take any $x\,\in\, X'$ and $v\, \in\, F'_x$. Let
$\widehat{v}$ be the unique flat section of $F'$, for the holomorphic connection 
$D'\vert_{X'}$, defined on a simply connected neighborhood $U\, \subset\, X'$ of $x$
such that $\widehat{v}(x)\,=\, v$. Consider the section $q'(\widehat{v})\, \in\,
H^0(U,\, Q)$, where $q'$ is the projection in \eqref{qp2}. Let
\begin{equation}\label{j2}
q'(\widehat{v})_1\, \in\, H^0(U,\, J^1(Q))
\end{equation}
be the section defined by $q'(\widehat{v})$. Let
\begin{equation}\label{j1}
q'(\widehat{v})^1_1\, \in\, H^0(U,\, J^1(J^1(Q)))
\end{equation}
be the section defined by $q'(\widehat{v})_1$.

Construct
$K(\widetilde{D^1})\, \subset\, J^1(J^1(Q))$ as in \eqref{kd} by substituting $D$ in
\eqref{kd} by the above logarithmic connection $D^1$. Similarly, define
$$
q^{D^1}_1\, :=\, (q_1\vert_{K(\widetilde{D^1})})\vert_{X'}\, :\, K(\widetilde{D^1})\vert_{X'}\,
\longrightarrow\, J^1(Q)\vert_{X'}
$$
as done in \eqref{e21}, where $q_1$ is the projection in \eqref{e17}.

Now from the construction of the logarithmic connection $D^1$ on $J^1(Q)$ we have
$$
q'(\widehat{v})^1_1\, \in\, H^0(U,\, K(\widetilde{D^1}))\, ,
$$
where $q'(\widehat{v})^1_1$ is constructed in \eqref{j1}, and moreover,
$$
q^{D^1}_1(q'(\widehat{v})^1_1)\,=\, q'(\widehat{v})_1\, \in\, H^0(U,\, J^1(Q))\, ,
$$
where $q'(\widehat{v})_1$ is constructed in \eqref{j2}. On the other hand,
we evidently have
$$
q_1(q'(\widehat{v})^1_1)\,=\, q_{0*}(q'(\widehat{v})^1_1)\, \in\, H^0(U,\, J^1(Q))\, ,
$$
where $q_{0*}$ is constructed in \eqref{e18}, because the section $q'(\widehat{v})^1_1$ of
$J^1(J^1(Q))$ is given by a section of $Q$ (namely, $q'(\widehat{v})$). In other words, we have
$((q_1- q_{0*})\vert_{X'})\circ (q^{D^1}_1)^{-1}\, =\, 0$. This proves \eqref{sd1} and the first part of the proposition.

To prove the second part of the proposition, we first note that the second jet bundle
fits in an exact sequence
$$
0\, \longrightarrow\, Q\otimes K^{\otimes 2}_X \,\stackrel{\iota_2}{\longrightarrow}\, J^2(Q)
\,\stackrel{q_2}{\longrightarrow}\, J^1(Q) \, \longrightarrow\, 0
$$
(see \eqref{je1}). The exact sequence in
\eqref{e17} and the homomorphism $q_1- q_{0*}$ in \eqref{e19} fit in the following commutative
diagram of homomorphisms
\begin{equation}\label{cd3}
\begin{matrix}
&& 0 && 0 && 0\\
&& \Big\downarrow && \Big\downarrow && \Big\downarrow\\
0 & \longrightarrow & Q\otimes K^{\otimes 2}_X & \stackrel{\iota_2}{\longrightarrow} & J^2(Q)
& \stackrel{q_2}{\longrightarrow} & J^1(Q) & \longrightarrow & 0\\
&& ~\,~\,~\,~\,~\,~\,~\,~\,~\,~\,~\,\Big\downarrow\iota_0\otimes{\rm Id}_{K_X} && ~ \Big\downarrow\mu && \Vert\\
0 & \longrightarrow & J^1(Q)\otimes K_X & \stackrel{\iota_1}{\longrightarrow} & J^1(J^1(Q))
& \stackrel{q_1}{\longrightarrow} & J^1(Q) & \longrightarrow & 0\\
&& ~\,~\,~\,~\,~\,~\,~\,~\,~\,~\,~\,~\,\Big\downarrow q_0\otimes{\rm Id}_{K_X} &&
~\,~\,~\,~\,~\,~\,~\,~\,~\,~\,\Big\downarrow q_1- q_{0*} \\
&& Q\otimes K_X & = & Q\otimes K_X\\
&& \Big\downarrow && \Big\downarrow && \\
&& 0 && 0
\end{matrix}
\end{equation}
of exact rows and columns, where the first vertical sequence is the exact
sequence in \eqref{e6} tensored with $K_X$ and $\mu$ is constructed in a
standard way.

Take any logarithmic connection
$$
D\, :\, J^1(J^1(Q))\, \longrightarrow\, J^1(Q)\otimes K_X\otimes {\mathcal O}_X(S)
$$
on $J^1(Q)$ such that $D\, \in\, {\mathcal C}^0(Q)$. The condition in Definition \ref{def1} that
$((q_1- q_{0*})\vert_{X'})\circ (q^D_1)^{-1}\, =\, 0$ implies that from \eqref{cd3} we have a commutative
diagram of exact columns
$$
\begin{matrix}
0 && 0\\
\Big\downarrow && \Big\downarrow\\
J^2(Q) & \stackrel{\mu_D}{\longrightarrow} & Q\otimes K^{\otimes 2}_X \otimes {\mathcal O}_X(S)\\
~\Big\downarrow\mu && ~\,~\,~\,~\,~\,~\,~\,~\,~\,\Big\downarrow\iota_0\otimes{\rm Id} \\
J^1(J^1(Q)) & \stackrel{ D}{\longrightarrow} & J^1(Q)\otimes K_X\otimes {\mathcal O}_X(S)\\
~\,~\,~\,~\,~\,~\,~\,~\,~\,~\,\Big\downarrow q_1- q_{0*} && 
~\,~\,~\,~\,~\,~\,~\,~\,~\,\Big\downarrow q_0\otimes {\rm Id}\\
Q \otimes K_X & \hookrightarrow & Q \otimes K_X\otimes {\mathcal O}_X(S)\\
\Big\downarrow && \Big\downarrow\\
0 && 0
\end{matrix}
$$
where $\text{Id}$ stands for the identity map of $K_X\otimes {\mathcal O}_X(S)$, while
$\iota_0$ and $q_0$ are the homomorphisms in \eqref{e6}; the above homomorphism
$\mu_D$ is uniquely defined by the commuting diagram. In other words,
\begin{equation}\label{md}
\mu_D\, \in\, H^0(X,\, \text{Diff}^2(Q,\, Q\otimes K^{\otimes 2}_X \otimes {\mathcal O}_X(S)))
\end{equation}
is a second order holomorphic differential operator.

Recall that the logarithmic connection on $\bigwedge^2 J^1(Q)\,=\, {\mathcal O}_X(S)$ induced by
$D$ coincides with the tautological connection ${\mathcal D}_S$ on ${\mathcal O}_X(S)$ given by the
de Rham differential (the third condition in the definition of ${\mathcal C}(Q)$). From this it
follows that the two second order differential operators given by any two elements of
${\mathcal C}^0(Q)$ differ by a $0$-th order differential operator, or in other words, they differ
by a holomorphic section of
$$
\text{Hom}(Q,\, Q\otimes K^{\otimes 2}_X \otimes {\mathcal O}_X(S))\,=\, K^{\otimes 2}_X \otimes {\mathcal O}_X(S)\, ;
$$
this is elaborated in Section \ref{se7}.

Conversely, take any surjective homomorphism
$$
\theta\, :\, J^2(Q)\, \longrightarrow\, Q\otimes K^{\otimes 2}_X\otimes {\mathcal O}_X(S)\, \longrightarrow\, 0
$$
such that the composition
$$
\theta\circ \iota_2\, :\, Q\otimes K^{\otimes 2}_X\, \longrightarrow\,
Q\otimes K^{\otimes 2}_X\otimes {\mathcal O}_X(S)
$$
coincides with the natural inclusion of the sheaf $Q\otimes K^{\otimes 2}_X$ in $Q\otimes K^{\otimes 2}_X
\otimes {\mathcal O}_X(S)$, where $\iota_2$ is the homomorphism in \eqref{cd3}.

Let ${\mathcal K}_\theta\, :=\, \text{kernel}(\theta)\, \subset\, J^2(Q)$ be the kernel.
The quotient $${\mathcal V}\, :=\, J^1(J^1(Q))/\mu({\mathcal K}_\theta)\, ,$$
where $\mu$ is the homomorphism in \eqref{cd3}, fits in the following exact sequence:
$$
0\, \longrightarrow\, Q\otimes K^{\otimes 2}_X\otimes {\mathcal O}_X(S)
\, \longrightarrow\, {\mathcal V}\, \stackrel{q_1-q_{0*}}{\longrightarrow}\, Q\otimes K_X
\, \longrightarrow\, 0\, ,
$$
where $q_1-q_{0*}$ is the restriction of the homomorphism $q_1-q_{0*}$ in \eqref{cd3}; in fact,
the above exact sequence fits in the commutative diagram
\begin{equation}\label{cdl}
\begin{matrix}
0 & \longrightarrow & Q\otimes K^{\otimes 2}_X &
\stackrel{\iota_0\otimes{\rm Id}_{K_X}}{\longrightarrow} & J^1(Q)\otimes K_X &
\stackrel{q_0\otimes{\rm Id}_{K_X}}{\longrightarrow} & Q\otimes K_X & \longrightarrow & 0\\
&& ~\, \Big\downarrow \widehat{\iota} && ~\Big\downarrow\varpi && \Vert\\
0 & \longrightarrow & Q\otimes K^{\otimes 2}_X\otimes {\mathcal O}_X(S)
& \longrightarrow & {\mathcal V} & \stackrel{q_1-q_{0*}}{\longrightarrow} & Q\otimes K_X
& \longrightarrow & 0
\end{matrix}
\end{equation}
where $\varpi$ is given by the inclusion $J^1(Q)\otimes K_X\, \hookrightarrow\, J^1(J^1(Q))$
(see \eqref{je1}) and
the map $\widehat{\iota}$ is the natural inclusion of sheaves. From this it follow that
$$
{\mathcal V}\, \hookrightarrow\, J^1(Q)\otimes K_X\otimes {\mathcal O}_X(S)\, .
$$
Now consider the composition homomorphism
\begin{equation}\label{cdo}
J^1(J^1(Q))\, \longrightarrow\, J^1(J^1(Q))/\mu({\mathcal K}_\theta)\,=\,{\mathcal V}\,
\hookrightarrow\, J^1(Q)\otimes K_X\otimes {\mathcal O}_X(S)\, .
\end{equation}
The first order holomorphic differential operator defined by it is a logarithmic
connection on $J^1(Q)$ singular over $S$.

Now it follows that the space of special logarithmic connections ${\mathcal C}^0(Q)$ is
identified with the space of all surjective homomorphisms
$$
\theta\, :\, J^2(Q)\, \longrightarrow\, Q\otimes K^{\otimes 2}_X\otimes {\mathcal O}_X(S)\, \longrightarrow\, 0
$$
such that
\begin{itemize}
\item the composition
$$
\theta\circ \iota_2\, :\, Q\otimes K^{\otimes 2}_X\, \longrightarrow\,
Q\otimes K^{\otimes 2}_X\otimes {\mathcal O}_X(S)
$$
coincides with $\widehat{\iota}$ in \eqref{cdl}, and

\item the corresponding logarithmic connection on $J^1(Q)$ has the property that 
the induced logarithmic connection on $\bigwedge^2 J^1(Q)\,=\, {\mathcal O}_X(S)$ coincides with
the canonical logarithmic connection ${\mathcal D}_S$ on ${\mathcal O}_X(S)$ given by the de Rham
differential.
\end{itemize}
From this the second part of the proposition follows.
\end{proof}

Consider the action of ${\rm Aut}(J^1(Q))$ on ${\mathcal C}(Q)$ in Lemma \ref{lem2}.

\begin{corollary}\label{cor3}
The quotient space ${\mathcal C}(Q)/{\rm Aut}(J^1(Q))$ is canonically identified with
${\mathcal C}^0(Q)$.
\end{corollary}

\begin{proof}
From Proposition \ref{prop2}(3) and Proposition \ref{prop3}(2) it follows that the composition
$$
{\mathcal C}^0(Q)\, \hookrightarrow\, {\mathcal C}(Q)
\, \longrightarrow\, {\mathcal C}(Q)/{\rm Aut}(J^1(Q))
$$
is injective. This composition is surjective by Proposition \ref{prop3}(1).
\end{proof}

\begin{corollary}\label{cor4}
The space of all branched projective structures on $X$ with branching divisor $S$ is canonically identified
with a subset of ${\mathcal C}^0(Q)$ consisting of all logarithmic connections $D$ for which
the homomorphism $\rho(D, y)$ in Proposition \ref{propn1} vanishes for every $y\,\in\, S$.

This space ${\mathcal C}^0(Q)$ is an affine space over the vector space $H^0(X,\,K^{\otimes 2}_X\otimes {\mathcal O}_X(S))$ whose
complex dimension is $3g-3+d$.
\end{corollary}

\begin{proof}
From Theorem \ref{thm1} it follows that the space of all branched projective structures on $X$ with 
branching type $S$ is identified with the subset of the quotient space ${\mathcal C}(Q)/{\rm Aut}(J^1(Q))$
given by all logarithmic connections $D$ for which
the homomorphism $\rho(D, y)$ in Proposition \ref{propn1} vanishes for every $y\,\in\, S$. Now Proposition \ref{prop3} 
(2) and Corollary \ref{cor3} together complete the proof. The dimension count follows from the Riemann--Roch theorem.
\end{proof}

\begin{remark}\label{rem-sub}
For each point $y\, \in\, S$, the condition $\rho(D, y)\,=\, 0$ on ${\mathcal C}^0(Q)$ defines a hypersurface of
${\mathcal C}^0(Q)$ of
codimension one. Hence the space of all branched projective structures on $X$ with branching divisor $S$ is a
subspace of ${\mathcal C}^0(Q)$ of codimension $d$ at a generic point. 
Similarly, in Mandelbaum's work in \cite{M1} and \cite{M2}, there is a codimension one condition at each branch point that arises as an 
 ``integrability condition'' (also called indicial equation) for the Schwarzian equation at the poles (\textit{cf.} Lemma 1 of \cite{Hejhal} or
\S~11.2 of \cite{GKM}). 
 \end{remark}

\section{Second order differential operators}\label{se7}

Fix a holomorphic line bundle $Q$ on $X$ as in \eqref{dq}. Take a second order holomorphic
differential operator from $Q$ to $Q\otimes K^{\otimes 2}_X\otimes {\mathcal O}_X(S)$
$$
D\, \in\, H^0(X,\, \text{Diff}^2(Q,\, Q\otimes K^{\otimes 2}_X\otimes {\mathcal O}_X(S)))\, .
$$
Its symbol is
$$
\gamma_2(D)\, \in\, H^0(X,\, \text{Hom}(Q,\, Q\otimes {\mathcal O}_X(S)))
\,=\, H^0(X,\, {\mathcal O}_X(S))\, ,
$$
where $\gamma_2$ is the homomorphism in \eqref{symb}.

Assume that this symbol $\gamma_2(D)$ is the section of ${\mathcal O}_X(S)$ given by the
constant function $1$. We also assume that the homomorphism $D\, :\, J^2(Q)\, \longrightarrow\,
Q \otimes K^{\otimes 2}_X\otimes {\mathcal O}_X(S)$ is surjective.

Take any $x\, \in\, X' =\, X\setminus S$ (see \eqref{e20}). Let $z$ be a holomorphic coordinate function on
a neighborhood $U\, \subset\, X'$ of $x$. Note that the restriction ${\mathcal O}_X(S)\vert_{X'}$
is the trivial holomorphic line bundle equipped with the holomorphic trivialization given by the
constant function $1$. So, we have
$$
(Q\vert_{X'})^{\otimes 2}\,=\, TX'\, .
$$
Let $(\partial_z)^{1/2}$ denote a holomorphic section of $Q\vert_U$ such that the
section $(\partial_z)^{1/2}\otimes (\partial_z)^{1/2}$ of $Q^{\otimes 2}\vert_U$ coincides with the
section $\frac{\partial}{\partial z}$ of $TU$. Note that the section $(\partial_z)^{1/2}$ produces a
holomorphic trivialization of $Q\vert_U$ because it does not vanish on any point of $U$.
Clearly, $(\partial_z)^{1/2}\otimes (dz)^{\otimes 2}$ is a
holomorphic section of $(Q\otimes K^{\otimes 2}_X)\vert_U$ which
does not vanish on any point of $U$. Consider the restriction $D\vert_U$ of the
differential operator $D$. For any holomorphic function $f$ on $U$, the holomorphic section
$D\vert_U(f\cdot (\partial_z)^{1/2})$ of $(Q\otimes K^{\otimes 2}_X)\vert_U$ is of the form
$$
D\vert_U(f\cdot (\partial_z)^{1/2})\,=\,
(\frac{d^2f}{dz^2} + a\cdot \frac{df}{dz} +b)(\partial_z)^{1/2}\otimes (dz)^{\otimes 2}\, ,
$$
where $a$ and $b$ are fixed holomorphic functions on $U$ independent of $f$; the coefficient of
$\frac{d^2f}{dz^2}$ is $1$ because the symbol of $D\vert_U$ is the constant function $1$.

It is straight-forward to check that the condition that
\begin{equation}\label{a}
a\,=\, 0
\end{equation}
is independent of the holomorphic coordinate function $z$. In other words, if we replace $z$ by
another holomorphic coordinate function $z_1$ on $U$, then $(\partial_z)^{1/2}$ will change,
hence $a$ and $b$ will change. But if $a\,=\, 0$ for one coordinate function, then $a$ vanishes for
all coordinate functions. Therefore, the condition that $a\,=\, 0$ is well-defined. We will
explain this condition in \eqref{a} intrinsically without using coordinates.

First recall from \eqref{cdo} that the differential operator $D$ gives a logarithmic connection
on $J^1(Q)$ singular over $S$. We will denote this logarithmic connection by $\widetilde{D}$.
Let $\widetilde{D}^{\rm det}$ be the logarithmic connection on $\bigwedge^2 J^1(Q)\,=\,
{\mathcal O}_X(S)$ induced by $\widetilde{D}$. On the other hand, ${\mathcal O}_X(S)$ has the
tautological logarithmic connection ${\mathcal D}_S$ given by the de Rham differential.
Therefore, we have
$$
\widetilde{D}^{\rm det} - {\mathcal D}_S\, \in\, H^0(X, \, K_X\otimes {\mathcal O}_X(S))\, .
$$
Since the residues of $\widetilde{D}^{\rm det}$ and ${\mathcal D}_S$ coincides (both are
$-1$ at each point of $S$), it follows that
$$
\widetilde{D}^{\rm det} - {\mathcal D}_S\, \in\, H^0(X, \, K_X)\, \subset\,
H^0(X, \, K_X\otimes {\mathcal O}_X(S))\, .
$$
Moreover, $a$ in \eqref{a} satisfies the equation
\begin{equation}\label{condet}
(\widetilde{D}^{\rm det} - {\mathcal D}_S)\vert_U\,=\, a\cdot dz\, .
\end{equation}
Therefore, the condition in \eqref{a} holds if and only if the two logarithmic connections
$\widetilde{D}^{\rm det}$ and ${\mathcal D}_S$ coincide.
In particular, the holomorphic one-form $a\cdot dz$ on $U$ is independent of the choice of the holomorphic
coordinate function $z$.

\begin{lemma}\label{lem4}
The space of all branched projective structures on $X$ with branching type $S$ is canonically identified
with a subset of the space of all second order holomorphic differential operators
$$
D\, \in\, H^0(X,\, {\rm Diff}^2(Q,\, Q\otimes K^{\otimes 2}_X\otimes {\mathcal O}_X(S)))
$$
satisfying the following three conditions:
\begin{enumerate}
\item the homomorphism $D\, :\, J^2(Q)\, \longrightarrow\,
Q \otimes K^{\otimes 2}_X\otimes {\mathcal O}_X(S)$ is surjective,

\item the symbol of $D$ is the section of ${\mathcal O}_X(S)$ given by the constant function $1$, and

\item the condition in \eqref{a} holds.
\end{enumerate}
\end{lemma}

\begin{proof}
In the proof of Proposition \ref{prop3} it was shown that ${\mathcal C}^0(Q)$
is identified with the space of all holomorphic differential operators
$$
D\, \in\, H^0(X,\, \text{Diff}^2(Q,\, Q\otimes K^{\otimes 2}_X\otimes {\mathcal O}_X(S)))
$$
satisfying the three conditions in the statement of the lemma. Note that for any logarithmic
connection $D\, \in\, {\mathcal C}^0(Q)$ on $J^1(Q)$, the connection on
$\bigwedge^2 J^1(Q)\,=\, {\mathcal O}_X(S)$ induced by $D$ coincides with the one given by
the de Rham differential. Therefore, from \eqref{condet} if follows that for the
second order differential operator corresponding to $D$, the condition
in \eqref{a} holds. Now the lemma follows from Corollary \ref{cor4}.
\end{proof}

%%%%%%%%%%%%%%%%%%%%%%%%%%%%%%%%%%%%%%%%%%%%%%%%%%%%%%%%%%%%%%%%%%%%%%%%%%%%%%%%%%%%%%%%%%

\end{document}